\documentclass[3p,10pt]{elsarticle}

\usepackage{amssymb,latexsym,amsmath,color,subfigure,amsthm}

\journal{}

\newcommand{\set}[1]{\left\{#1\right\}}
\newcommand{\abs}[1]{\left|#1\right|}
\newcommand{\p}{\partial}

\newcommand{\mc}{\mathbf{c}}
\newcommand{\md}{\mathbf{d}}

\newcommand{\mt}{\mathbf{t}}
\newcommand{\mx}{\mathbf{x}}

\newcommand{\vt}{\boldsymbol{\theta}}
\newcommand{\vv}{\boldsymbol{\varphi}}

\theoremstyle{plain}
\newtheorem{thm}{Theorem}[section]
\newtheorem{lem}[thm]{Lemma}
\newtheorem{cor}[thm]{Corollary}

\theoremstyle{remark}
\newtheorem{rem}{Remark}[section]
\newtheorem{ex}{Example}[section]

\begin{document}

\begin{frontmatter}



\title{Direct sampling method for retrieving small perfectly conducting cracks}


\author{Won-Kwang Park}
\ead{parkwk@kookmin.ac.kr}
\address{Department of Information Security, Cryptology, and Mathematics, Kookmin University, Seoul, 02707, Korea}

\begin{abstract}
We consider direct sampling method for finding location of a set of linear perfectly conducting cracks with small length from collected far-field data corresponding to the single incident field. To show the feasibility of direct sampling method, we first prove that the indicator function of direct sampling method can be represented by the Bessel function of order zero and the length of cracks. Results of numerical simulations are shown to support the fact that the imaging performance is highly depending on the length of cracks. To explain the fact that imaging performance is highly depending on the rotation of crack, we perform further analysis of direct sampling method by establishing a representation by the Bessel functions of order zero and one.
\end{abstract}


\end{frontmatter}





\section{Introduction}
This work concerned on direct sampling method for a fast imaging of small, linear perfectly conducting cracks located in two-dimensional space $\mathbb{R}^2$. It is well-known that direct sampling method is a fast, simple and effective imaging technique. Furthermore, it requires only a few (one or two) incident fields and does not requires additional operations (e.g., singular value decomposition, solving adjoint problems or ill-posed integral equations, etc.). Due to this reason, it applied to many inverse scattering problems \cite{CILZ,CIZ,IJZ1,IJZ2,LZ}.

Based on these studies, it turns out that direct sampling method is an effective in full-view inverse scattering problem. Specially, based on the relationship between Bessel function of order zero and the indicator function of direct sampling method \cite{IJZ1,LZ}, the reason of detection of targets has been investigated. However, the analysis is not fully reliable in the imaging of cracks, for example, cracks whose lengths are significantly smaller than those of the others are theoretically undetectable and identified location is different corresponding to the direction of propagation. Hence, a further analysis of indicator function of direct sampling method still needs to be performed, which is the motivation for our work.

In this paper, we carefully identify mathematical structure of indicator function of direct sampling method. In detail, we prove that the indicator function can be represented by the Bessel functions of order zero and one, length and rotation of cracks, and the direction of propagation. This is based on the fact that the far-field pattern can be represented by the asymptotic expansion formula in the presence of small, linear perfectly conducting cracks (see \cite{AKLP} for instance). From the identified structure, we explain the reason of unexplained phenomenon and find two methods of improvement by applying multiple incident fields and frequencies. Throughout careful analysis and numerical experiments, we demonstrate the improvement of direct sampling method theoretically and numerically.

This paper is organized as follows. In Section \ref{sec:2}, we survey two-dimensional forward problem, asymptotic expansion formula due to the existence of small cracks, and indicator function of direct sampling method. In Section \ref{sec:3}, we carefully identify the structure of indicator function by establishing a relationship with Bessel functions of order zero and one, length and rotation of cracks, and the incident field direction to explain the identification of direct sampling method is highly depending on not only the length and rotation of cracks but also the direction of incident field. To support identified structure, several results of numerical simulations exhibited. In Section \ref{sec:4}, we introduce two methods of improvement by applying multiple directions of incident fields and multiple frequencies. Furthermore, we perform numerical simulations to examine the improvement. Section \ref{sec:5} contains a short conclusion and some remarks on future work.

\section{Forward problem and direct sampling method}\label{sec:2}
\subsection{Two-dimensional forward problem and far-field pattern}
In this section, we introduce the two-dimensional direct scattering problem for $M$ different, well-separated linear perfectly conducting cracks of length $2\ell_m$, denoted by $\Sigma_m$, $m=1,2,\cdots,M$, located in the homogeneous space $\mathbb{R}^2$. For a more detailed description, we recommend \cite{K}. Throughout this study, we denote $\Sigma_m$ as
\[\Sigma_m=\set{\mc_m=\mathcal{R}_{\phi}[x_m,y_m]^\mathrm{T}:-\ell_m\leq x_m\leq\ell_m},\]
for $m=1,2,\cdots,M$, and let $\Sigma$ be the collection of cracks. Here $\mc_m$ is the center of $\Sigma_m$ and $\mathcal{R}_{\phi}$ denotes rotation by $\phi$. We assume that the $\Sigma_m$ are sufficiently separated from each other such that
\[k|\mc_m-\mc_{m'}|\gg1-\frac14=\frac34,\]
where $k$ denotes the positive wavenumber, which is of the form $k=2\pi/\lambda$. Here, $\lambda$ is the given wavelength and assume that $2\ell_m\ll\lambda$ and $k\ell_m\rightarrow0+$ for all $m=1,2,\cdots,M$. In this study, following from \cite{LZ}, we consider the plane-wave illumination: let $\psi_{\mathrm{inc}}(\mx,\md)=e^{ik\md\cdot\mx}$ be the given incident field with fixed propagation direction $\md\in\mathbb{S}^1$. Here $\mathbb{S}^1$ denotes the two-dimensional unit circle centered at the origin. And let $\psi(\mx,\md)$ be the time-harmonic total field that satisfies the following Helmholtz equation
\begin{equation}\label{HelmholtzEquation}
  \triangle \psi(\mx,\md)+k^2\psi(\mx,\md)=0\quad\mbox{in}\quad\mathbb{R}^2\backslash\overline{\Sigma}
\end{equation}
with Dirichlet boundary condition
\begin{equation}\label{BoundaryCondition}
  \psi(\mx,\md)=0\quad\mbox{on}\quad\Sigma.
\end{equation}

Let $\psi_{\mathrm{scat}}(\mx,\md)=\psi(\mx,\md)-\psi_{\mathrm{inc}}(\mx,\md)$ be the scattered field $\psi_{\mathrm{scat}}(\mx,\md)$ that satisfy the Sommerfeld radiation condition
\[\lim_{\abs{\mx}\to\infty}\sqrt{\abs{\mx}}\left(\frac{\p \psi_{\mathrm{scat}}(\mx,\md)}{\p\abs{\mx}}-ik\psi_{\mathrm{scat}}(\mx,\md)\right)=0\]
uniformly in all directions $\vt=\mx/\abs{\mx}$. We denote $\psi_\infty(\vt,\md)$ as the far-field pattern of the $\psi_{\mathrm{scat}}(\mx,\md)$ that satisfies
\[\psi_{\mathrm{scat}}(\mx,\md)=\frac{e^{ik|\mx|}}{\sqrt{|\mx|}}\left\{\psi_\infty(\vt,\md)+\mathcal{O}\left(\frac{1}{|\mx|}\right)\right\},\quad|\mx|\longrightarrow+\infty\]
uniformly in all directions $\vt=\mx/|\mx|\in\mathbb{S}^1$. Based on \cite{K}, $\psi_\infty(\vt,\md)$ can be represented as the following single-layer potential with unknown density function $\varphi(\mc_m,\md)$:
\begin{equation}\label{FarFieldPattern}
\psi_\infty(\vt,\md)=-\frac{1+i}{4\sqrt{\pi k}}\sum_{m=1}^{M}\int_{\Sigma_m}e^{-ik\vt\cdot\mc_m}\varphi(\mc_m,\md)d\mc_m.
\end{equation}
Based on \cite{AKLP}, the far-field pattern $\psi_\infty(\vt,\md)$ can be represented as the following asymptotic expansion formula, which plays a key role in the analysis of the imaging function of the direct sampling method.

\subsection{Indicator function of direct sampling method}
Now, we briefly introduce the indicator function of direct sampling method for finding location of $\Sigma_m$ from a set of measured far-field pattern data
\[\Psi:=\set{\psi_\infty(\vt_n,\md):n=1,2,\cdots,N}.\]
Throughout this paper, we assume that total number of $N$ is sufficiently large and consider the full-view inverse scattering problem, i.e., we set
\[\vt_n=\bigg[\cos\frac{2\pi n}{N},\sin\frac{2\pi n}{N}\bigg]^T.\]
Then, for a search point $\mx\in\mathbb{R}^2$, the indicator function of direct sampling method is given by
\begin{equation}\label{ImagingFunction}
  \mathcal{I}(\mx):=\frac{|\langle \psi_\infty(\vt_n,\md),e^{-ik\vt_n\cdot\mx}\rangle|}{||\psi_\infty(\vt_n,\md)||_{L^2(\mathbb{S}^1)}||e^{-ik\vt_n\cdot\mx}||_{L^2(\mathbb{S}^1)}},
\end{equation}
where
\[\langle f_1,f_2\rangle:=\sum_{n=1}^{N}f_1\overline{f}_2\quad\mbox{and}\quad||f||_{L^2(\mathbb{S}^1)}=\sqrt{\langle f,f\rangle}.\]
Following \cite{LZ}, it has been confirmed that $\mathcal{I}(\mx)$ satisfies the relation
\begin{equation}\label{TraditionalDSM}
\mathcal{I}(\mx)\approx\sum_{m=1}^{M}J_0(k|\mx-\mc_m|).
\end{equation}
This means that $\mathcal{I}(\mx)$ plots peaks of magnitude $1$ at $\mx=\mc_m$ and has small magnitude elsewhere so that location of $\mc_m$ can be identified via the map of $\mathcal{I}(\mx)$. Here $J_0$ denotes the Bessel function of the first kind of order zero.

On the basis of the relation (\ref{TraditionalDSM}), the feasibility of direct sampling method can be explained. However, following two phenomenon can be observed through the simulation but the the reason of phenomenon not be explained theoretically:
\begin{enumerate}
\item the value of $\mathcal{I}(\mx)$ is highly depending on the length $\ell_m$ of $\Sigma_m$, refer to Figure \ref{DSM-Length2}.
\item if $\ell_m$ are same, the value of $\mathcal{I}(\mx)$ is highly depending on the rotation $\mathcal{R}_\phi$, refer to Figure \ref{DSM-Incident}.
\end{enumerate}
Motivated by this, we carefully analyze the indicator function to explain unexpected results.

\section{Structure analysis of indicator function}\label{sec:3}
\subsection{Analysis of indicator function: dependency of the length of cracks}\label{sec3-1}
First, we explore the structure of indicator function by establishing a relationship with Bessel function of order zero and length of cracks. For this, we adopt an asymptotic expansion formula due to the presence of $\Sigma_m$, refer to \cite{AKLP}. This plays a key role of our analysis.

\begin{lem}[Asymptotic expansion formula]\label{LemmaAsymptotic1} If $\psi(\mx,\md)$ satisfies (\ref{HelmholtzEquation}) and (\ref{BoundaryCondition}), and $\psi_{\mathrm{inc}}(\mx,\md)=e^{ik\md\cdot\mx}$, then following asymptotic expansion formula holds for $0<\ell_m<2$ and $\ell_m\ll\lambda/2$:
  \begin{equation}\label{Asymptotic1}
    \psi_\infty(\vt,\md)=\sum_{m=1}^{M}\frac{2\pi}{\ln(\ell_m/2)}e^{ik\md\cdot\mc_m}e^{-ik\vt\cdot\mc_m}+\mathcal{O}\left(\frac{1}{|\ln\ell_m|^2}\right).
  \end{equation}
\end{lem}

Following result is useful to explore the structure. A rigorous derivation is in \cite{P-SUB3}.
\begin{lem}\label{TheoremBessel1}
  For a sufficiently large $N$, $\vt_n,\vt\in\mathbb{S}^1$, and $\mx\in\mathbb{R}^2$, the following relation holds:
  \[\sum_{n=1}^{N}e^{ik\vt_n\cdot\mx}=\int_{\mathbb{S}^1}e^{ik\vt\cdot\mx}d\vt=2\pi J_0(k|\mx|).\]
\end{lem}

By combining Lemmas \ref{LemmaAsymptotic1} and \ref{TheoremBessel1}, we can obtain the following structure of indicator function. The result is follows.

\begin{thm}[Structure of indicator function]\label{TheoremStructure1}
Assume that total number of observation direction $N$ is sufficiently large. Then, $\mathcal{I}(\mx)$ can be represented as
\begin{equation}\label{Structure1}
\mathcal{I}(\mx)\approx\abs{\sum_{m=1}^{M}\frac{J_0(k|\mx-\mc_m|)}{\ln(\ell_m/2)}}\left(\max\left|\sum_{m=1}^{M}\frac{1}{\ln(\ell_m/2)}\right|\right)^{-1}.
\end{equation}
\end{thm}
\begin{proof}
Applying (\ref{Asymptotic1}) to (\ref{ImagingFunction}), we can evaluate
\begin{align*}
  \langle \psi_\infty(\vt_n,\md),e^{-ik\vt\cdot\mx}\rangle&\approx\sum_{n=1}^{N}\sum_{m=1}^{M}\frac{2\pi}{\ln(\ell_m/2)}e^{ik\md\cdot\mc_m}e^{-ik\vt_n\cdot\mc_m}\overline{e^{-ik\vt_n\cdot\mx}}\\
  &=\sum_{m=1}^{M}\frac{2\pi}{\ln(\ell_m/2)}e^{ik\md\cdot\mc_m}\left(\sum_{n=1}^{N}e^{ik\vt_n\cdot(\mx-\mc_m)}\right).
\end{align*}
With this, by applying Lemma \ref{TheoremBessel1}, we can obtain
\[\langle \psi_\infty(\vt_n,\md),e^{-ik\vt_n\cdot\mx}\rangle\approx\sum_{m=1}^{M}\frac{(2\pi)^2}{\ln(\ell_m/2)}e^{ik\md\cdot\mc_m}J_0(k|\mx-\mc_m|).\]
Since, $||e^{-ik\vt_n\cdot\mx}||_{L^2(\mathbb{S}^1)}=1$, $|e^{ik\md\cdot\mc_m}|=1$, and $J_0$ has maximum value $1$, applying H{\"o}lder's inequality, we arrive (\ref{Structure1}). This completes the proof.
\end{proof}

\begin{rem}\label{Observation1}
Based on the identified structure (\ref{Structure1}), we can observe that the location of $\mc_m$ can be detected by plotting $\mathcal{I}(\mx)$. However, if the length of $\Sigma_m$ is significantly shorter than the others, its location will not be detectable. Hence, we can conclude that the imaging performance of direct sampling method is highly depending on the length of crack.
\end{rem}

\subsection{Numerical simulations: part 1}\label{sec3-2}
In this section, results of numerical simulations are presented to support identified structure of (\ref{Structure1}). For this, three linear cracks $\Sigma_m$ with lengths $2\ell_m$ were used throughout the numerical simulations:
\begin{align*}
  \Sigma_1&=\set{[s+0.6,0.2]^T:-\ell_1\leq s\leq\ell_1}\\
  \Sigma_2&=\set{\mathcal{R}_{\pi/4}[s-0.4,s-0.35]^T:-\ell_2\leq s\leq\ell_2}\\
  \Sigma_3&=\set{\mathcal{R}_{7\pi/6}[s-0.25,s+0.6]^T:-\ell_3\leq s\leq\ell_3}.
\end{align*}
The wavelength $\lambda$ was set to $0.5$ and $N=30$ elements of $\Psi$ were collected, where all elements of $\Psi$ were generated by solving the Fredholm integral equation of the second kind along the cracks $\Sigma_m$ introduced in \cite[Chapter 4]{N}.

\begin{ex}[Imaging of cracks with the same length]
Figure \ref{DSM-Traditional} shows map $\mathcal{I}(\mx)$ for $\md=[0,1]^T$ when the lengths of all $\Sigma_m$ are the same, say, $\ell_m\equiv0.05$. As for the traditional results \cite{LZ}, although there exists some artifacts but the exact locations $\mc_m$ can be identified clearly.
\end{ex}

\begin{figure}[h]
  \begin{center}
    \includegraphics[width=0.495\textwidth]{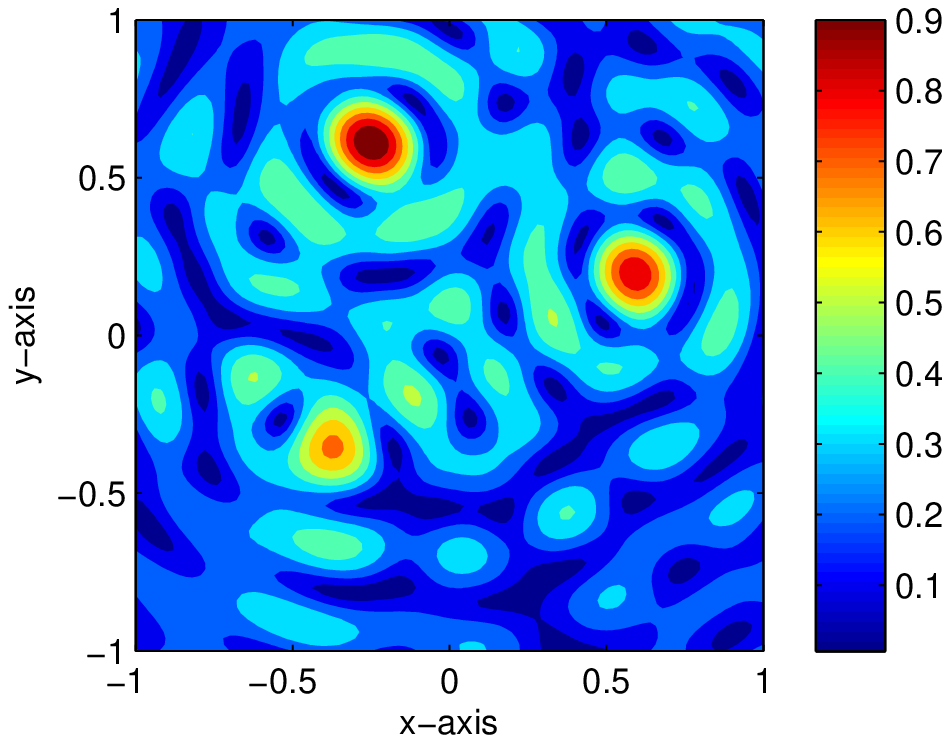}
    \includegraphics[width=0.495\textwidth]{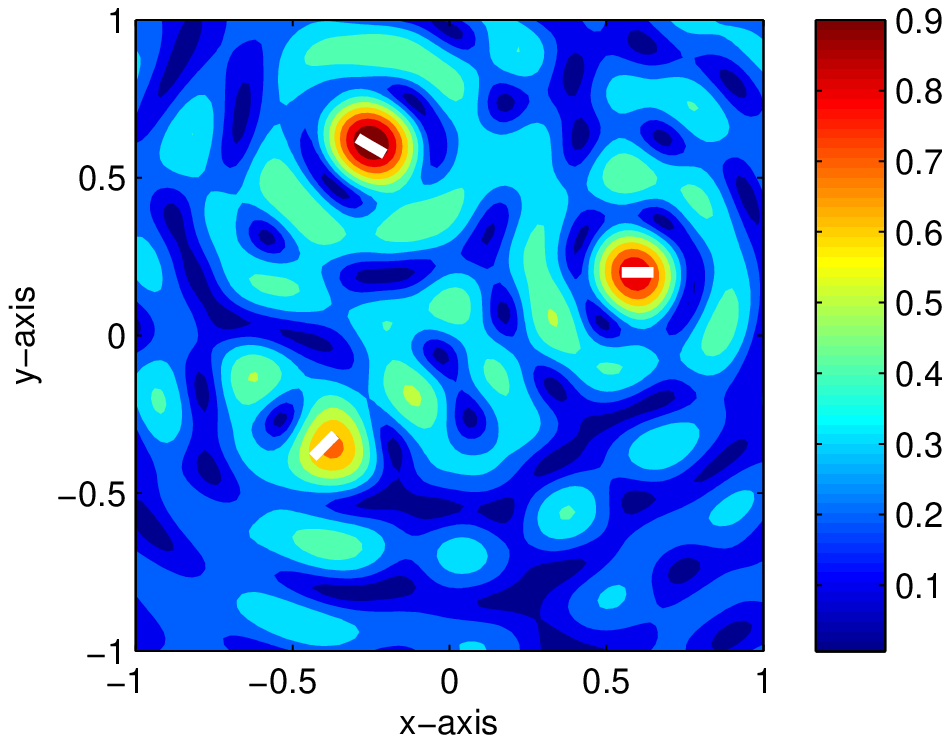}
    \caption{\label{DSM-Traditional}Map of $\mathcal{I}(\mx)$ for $\lambda=0.5$ when $\ell_m\equiv0.05$, $m=1,2,3$.}
  \end{center}
\end{figure}

\begin{ex}[Imaging of cracks with different lengths]
In this example, we consider the imaging of cracks when all the lengths are different. Figure \ref{DSM-Length1} shows map of $\mathcal{I}(\mx)$ for $\md=[0,1]^T$ when $\ell_1=0.05$, $\ell_2=0.04$, and $\ell_3=0.03$. Based on this result, we can observe that the value of $\mathcal{I}(\mc_3)$ is smaller than the values $\mathcal{I}(\mc_1)$ and $\mathcal{I}(\mc_2)$ because the length of $\Sigma_3$ is shorter than the others. Although some artifacts are exists in the map, we can identify locations of all cracks.
\end{ex}

\begin{figure}[h]
  \begin{center}
    \includegraphics[width=0.495\textwidth]{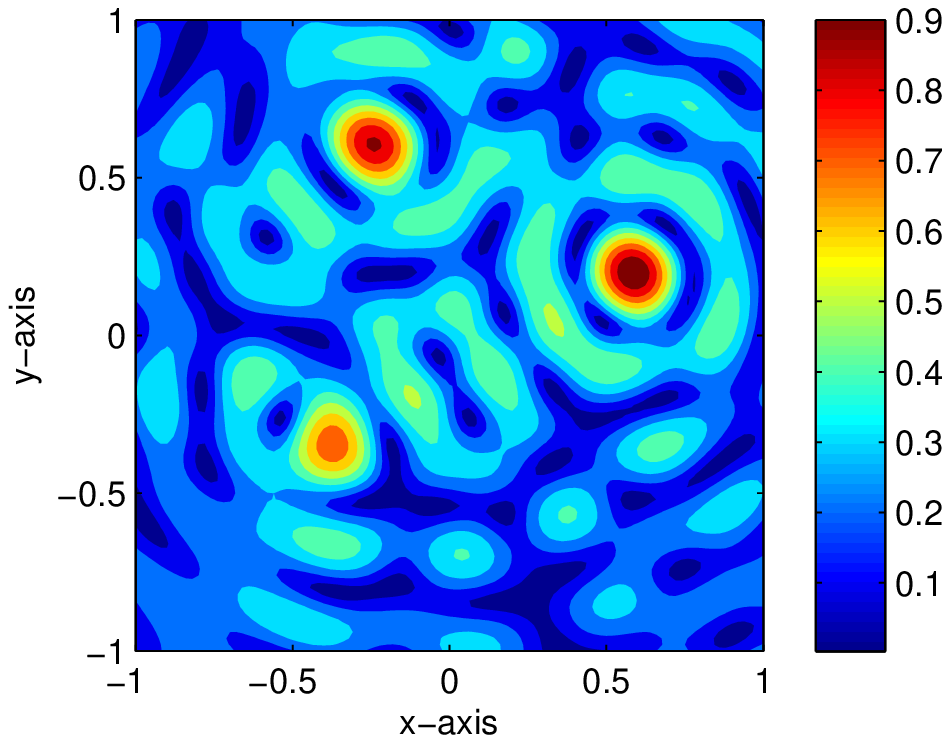}
    \includegraphics[width=0.495\textwidth]{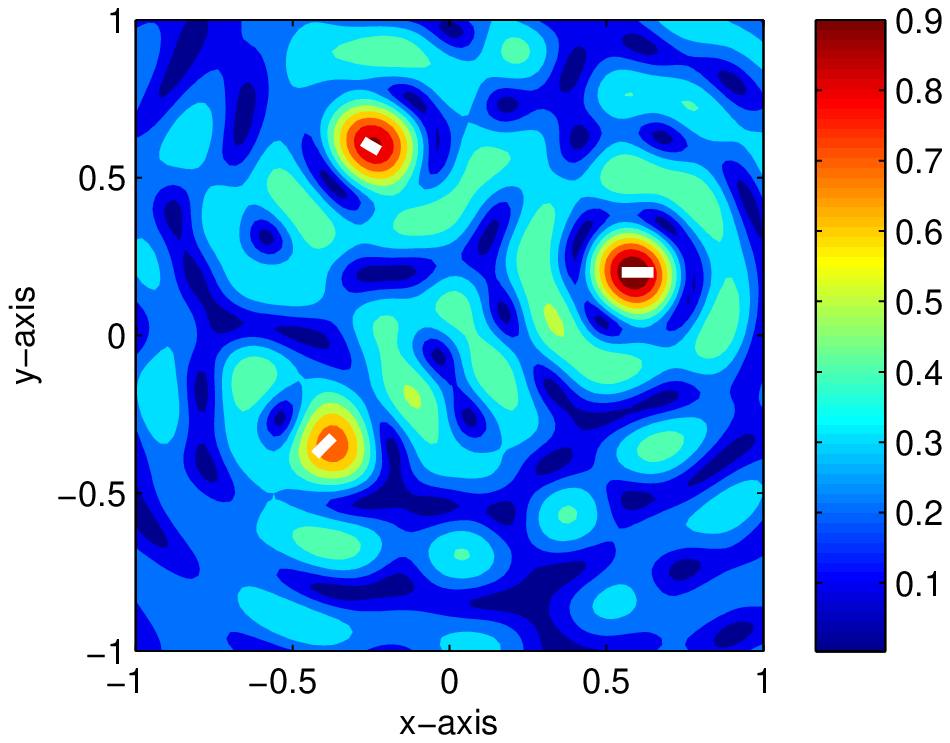}
    \caption{\label{DSM-Length1}Map of $\mathcal{I}(\mx)$ for $\lambda=0.5$ when $\ell_1=0.05$, $\ell_2=0.04$, and $\ell_3=0.03$.}
  \end{center}
\end{figure}

\begin{ex}[Imaging of cracks with extremely different lengths]
In this example, we consider the imaging of cracks when one crack is significantly longer than the others. Figure \ref{DSM-Length2} shows map of $\mathcal{I}(\mx)$ for $\md=[0,1]^T$ when $\ell_1=0.05$ and $\ell_2=\ell_3=0.01$. As discussed, only $\Sigma_1$ can be identified clearly via the map of $\mathcal{I}(\mx)$ because the lengths of the remaining cracks (here, $\Sigma_2$ and $\Sigma_3$) are significantly shorter than that of $\Sigma_1$. Unfortunately, it is very to identify the location of $\Sigma_2$ and $\Sigma_3$ due to the appearance of artifacts.
\end{ex}

\begin{figure}[h]
  \begin{center}
    \includegraphics[width=0.495\textwidth]{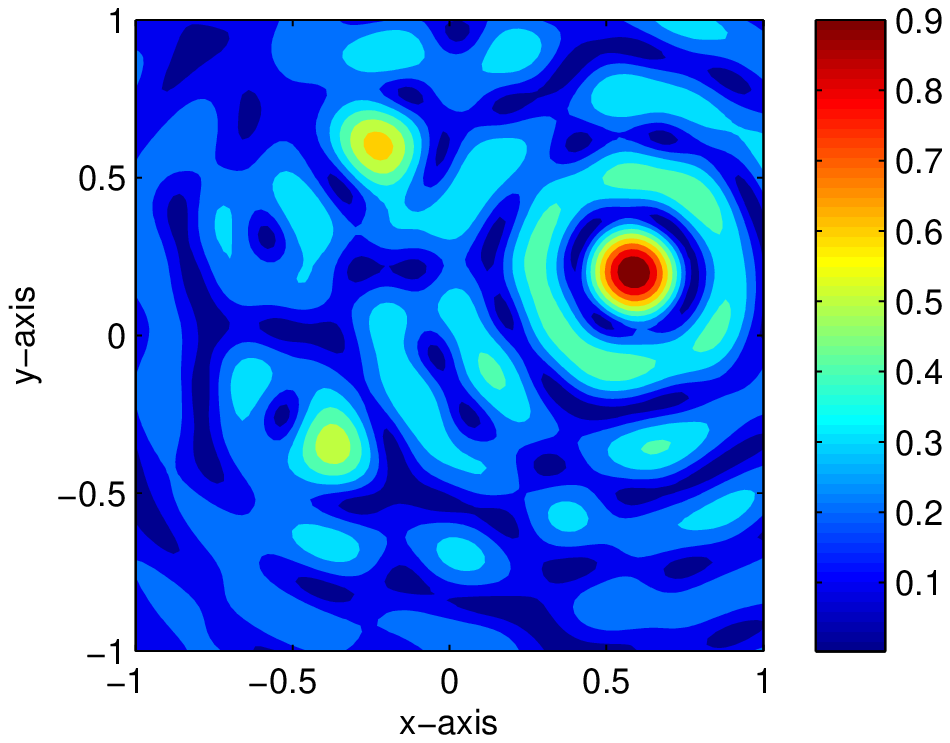}
    \includegraphics[width=0.495\textwidth]{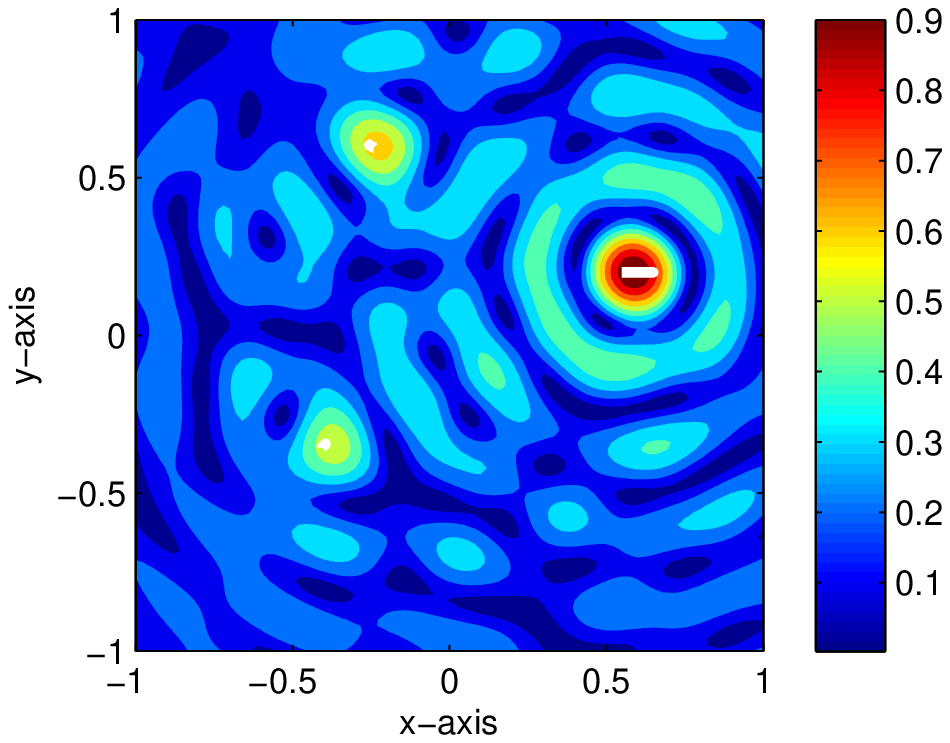}
    \caption{\label{DSM-Length2}Map of $\mathcal{I}(\mx)$ for $\lambda=0.5$ when $\ell_1=0.05$, $\ell_2=\ell_3=0.01$.}
  \end{center}
\end{figure}

\begin{ex}[Influence of incident direction]
In this example, we consider the influence of incident direction. Figure \ref{DSM-Incident} shows maps of $\mathcal{I}(\mx)$ for $\md=[\cos(3\pi/4),\sin(3\pi/4)]^T$ and $\md=[\cos(\pi/6),\sin(\pi/6)]^T$. It is interesting to observe that although three peaks of large magnitudes are appear in the map of $\mathcal{I}(\mx)$, identified locations are inaccurate. Furthermore, identified locations are shifted when applied incident direction $\md$ is varied. Unfortunately, we cannot explain this phenomenon via the structure (\ref{Structure1}). Hence, further analysis of indicator function is needed to explain this result.
\end{ex}

\begin{figure}[h]
  \begin{center}
    \includegraphics[width=0.495\textwidth]{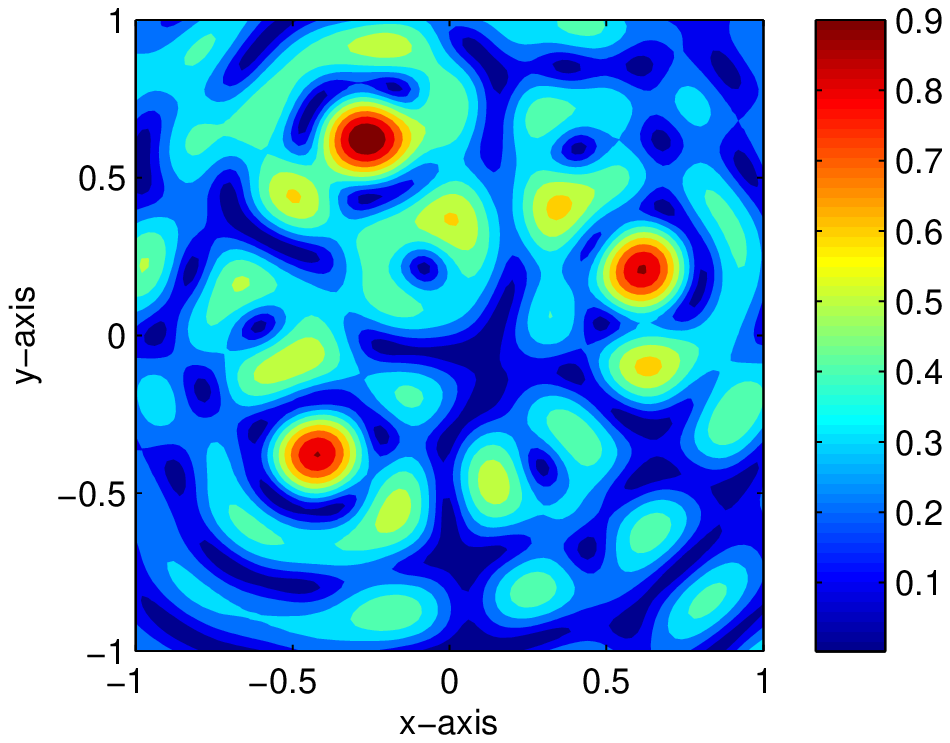}
    \includegraphics[width=0.495\textwidth]{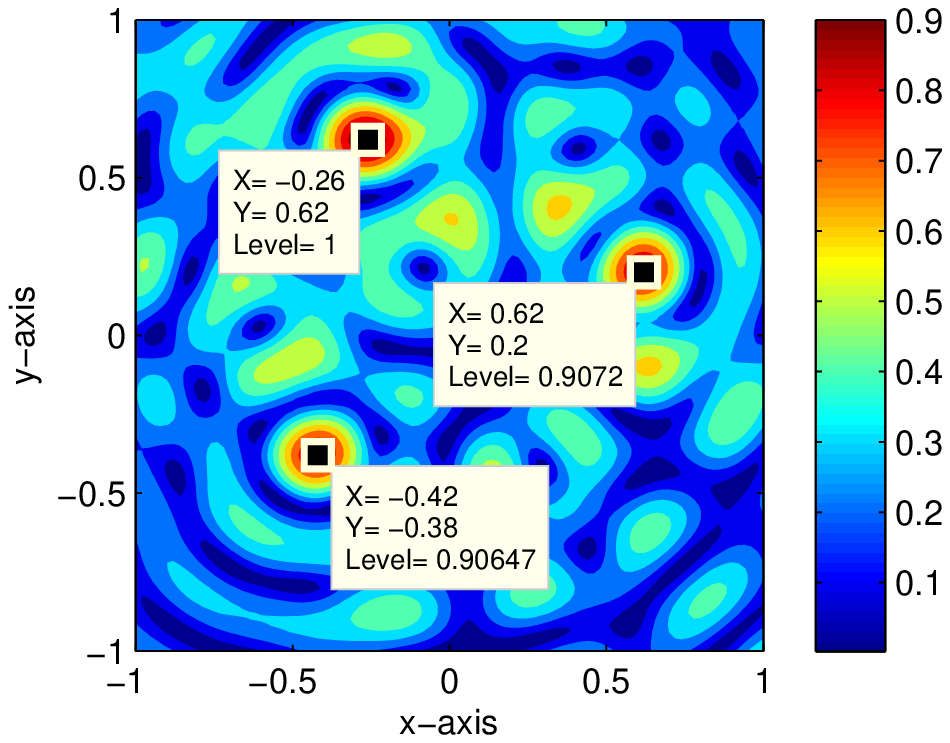}\\
    \includegraphics[width=0.495\textwidth]{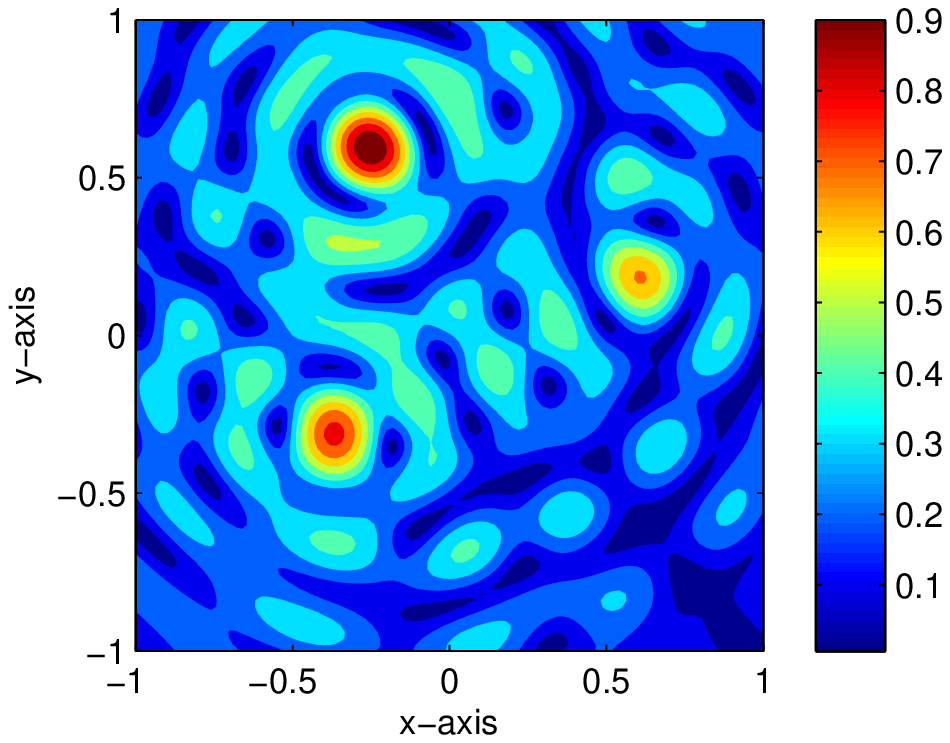}
    \includegraphics[width=0.495\textwidth]{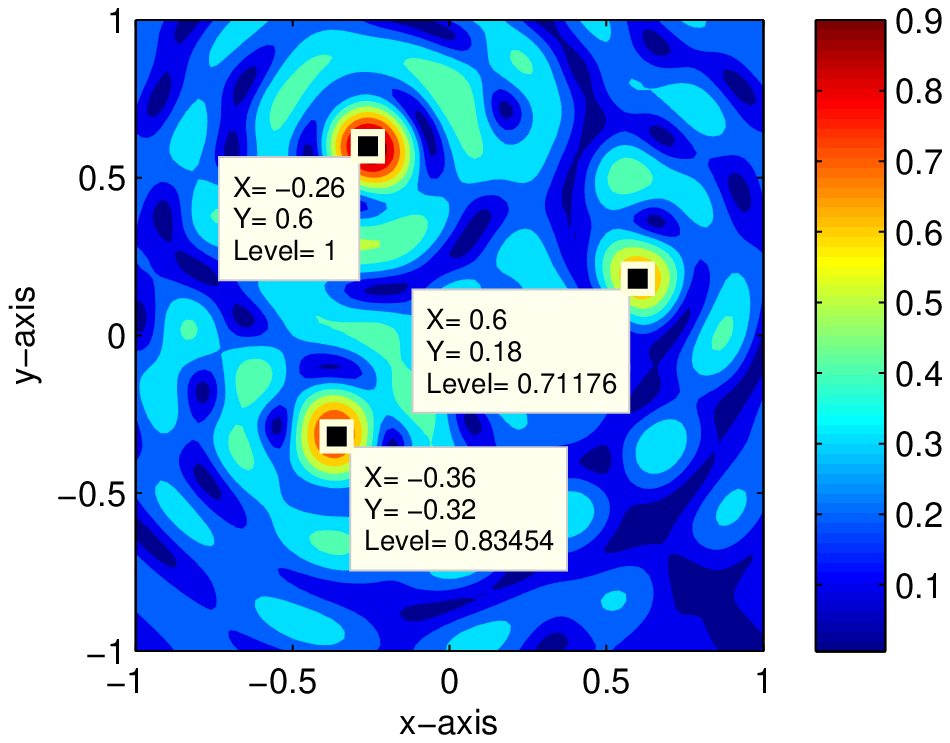}
    \caption{\label{DSM-Incident}Map of $\mathcal{I}(\mx)$ for $\md=[\cos(3\pi/4),\sin(3\pi/4)]^T$ (top) and $\md=[\cos(\pi/6),\sin(\pi/6)]^T$ (bottom).}
  \end{center}
\end{figure}

\subsection{Analysis of indicator function: dependency of the rotation of cracks}\label{sec3-3}
In order to explain the results in Figure \ref{DSM-Incident}, we explore the structure of indicator function by establishing a relationship with Bessel function of order zero and one, incident direction, and rotation of cracks. For this, we adopt second-order asymptotic expansion formula due to the presence of $\Sigma_m$, refer to \cite{AKLP}. In this section, we assume that $\ell_m\equiv\ell$ for $m=1,2,\cdots,M$.

\begin{lem}[Asymptotic expansion formula: higher order]\label{LemmaAsymptotic2} If $\psi(\mx,\md)$ satisfies (\ref{HelmholtzEquation}) and (\ref{BoundaryCondition}), and $\psi_{\mathrm{inc}}(\mx,\md)=e^{ik\md\cdot\mx}$, then following asymptotic expansion formula holds for $0<\ell<2$ and $\ell\ll\lambda/2$:
  \begin{equation}\label{Asymptotic2}
    \psi_\infty(\vt,\md)=\sum_{m=1}^{M}\bigg(\frac{2\pi}{\ln(\ell/2)}e^{ik\md\cdot\mc_m}e^{-ik\vt\cdot\mc_m}-\pi\ell^2\frac{\p e^{ik\md\cdot\mc_m}}{\p\mt(\mc_m)}\frac{\p e^{-ik\vt\cdot\mc_m}}{\p\mt(\mc_m)}\bigg)+\mathcal{O}(\ell^3),
  \end{equation}
  where $\p/\p\mt(\mc_m)$ denotes the tangential derivative at $\mc_m\in\Sigma_m$.
\end{lem}

And we introduce a useful relation derived in \cite{P-SUB3}.
\begin{lem}\label{TheoremBessel2}
  For a sufficiently large $N$, $\vt_n,\vt,\vv\in\mathbb{S}^1$, and $\mx\in\mathbb{R}^2$, the following relation holds:
  \[\sum_{n=1}^{N}(\vv\cdot\vt_n)e^{ik\vt_n\cdot\mx}=\int_{\mathbb{S}^1}(\vv\cdot\vt)e^{ik\vt\cdot\mx}d\vt=2\pi i\left(\frac{\mx}{|\mx|}\cdot\vv\right)J_1(k|\mx|).\]
\end{lem}

By combining Lemmas \ref{LemmaAsymptotic2} and \ref{TheoremBessel2}, we can obtain the following structure of indicator function. The result is follows.

\begin{thm}[Structure of indicator function]\label{TheoremStructure2}
Assume that total number of observation direction $N$ is sufficiently large. Then, $\mathcal{I}(\mx)$ can be represented as
\begin{equation}\label{Structure2}
\mathcal{I}(\mx)=\frac{|\Phi_1(\mx)+\Phi_2(\mx,\md)|}{\displaystyle\max_{\mx\in\mathbb{R}^2}|\Phi_1(\mx)+\Phi_2(\mx,\md)|},
\end{equation}
where
\begin{align}
\Phi_1(\mx,\md)&=\frac{2}{\ln(\ell/2)}\sum_{m=1}^{M}e^{ik\md\cdot\mc_m}J_0(k|\mx-\mc_m|)\label{firstterm}\\
\Phi_2(\mx,\md)&=-k^2\ell^2 i\sum_{m=1}^{M}\bigg((\md\cdot\mt(\mc_m))e^{ik\md\cdot\mc_m}\bigg)\bigg(\frac{\mx-\mc_m}{|\mx-\mc_m|}\cdot\mt(\mc_m)\bigg)J_1(k|\mx-\mc_m|).\label{secondterm}
  \end{align}
\end{thm}
\begin{proof}
Based on Lemma \ref{TheoremStructure2} and an elementary calculus, we can evaluate
\begin{align*}
\sum_{n=1}^{N}\frac{\p e^{ik\md\cdot\mc_m}}{\p\mt(\mc_m)}\frac{\p e^{-ik\vt_n\cdot\mc_m}}{\p\mt(\mc_m)}\overline{e^{-ik\vt_n\cdot\mx}}&=\sum_{n=1}^{N}\bigg((ik\md\cdot\mt(\mc_m))e^{ik\md\cdot\mc_m}\bigg)\bigg((ik\vt_n\cdot\mt(\mc_m))e^{ik\vt_n\cdot(\mx-\mc_m)}\bigg)\\
&=-k^2\bigg((\md\cdot\mt(\mc_m))e^{ik\md\cdot\mc_m}\bigg)\bigg(\sum_{n=1}^{N}(\vt_n\cdot\mt(\mc_m))e^{ik\vt_n\cdot(\mx-\mc_m)}\bigg)\\
&=-2\pi ik^2\bigg((\md\cdot\mt(\mc_m))e^{ik\md\cdot\mc_m}\bigg)\bigg(\frac{\mx-\mc_m}{|\mx-\mc_m|}\cdot\mt(\mc_m)\bigg)J_1(k|\mx-\mc_m|)
\end{align*}
With this, applying (\ref{Asymptotic2}) to (\ref{ImagingFunction}) and Theorem \ref{TheoremStructure1}, we can obtain
\begin{align*}
  \langle\psi_\infty(\vt_n,\md),e^{-ik\vt\cdot\mx}\rangle\approx&\sum_{n=1}^{N}\sum_{m=1}^{M}\left(\frac{2\pi}{\ln(\ell/2)}e^{ik\md\cdot\mc_m}e^{-ik\vt_n\cdot\mc_m}\overline{e^{-ik\vt_n\cdot\mx}}-\pi\ell^2\frac{\p e^{ik\md\cdot\mc_m}}{\p\mt(\mc_m)}\frac{\p e^{-ik\vt_n\cdot\mc_m}}{\p\mt(\mc_m)}\overline{e^{-ik\vt_n\cdot\mx}}\right)\\
  =&\frac{(2\pi)^2}{\ln(\ell/2)}\sum_{m=1}^{M}e^{ik\md\cdot\mc_m}J_0(k|\mx-\mc_m|)\\
  &-2\pi^2 k^2 \ell^2 i\sum_{m=1}^{M}\bigg((\md\cdot\mt(\mc_m))e^{ik\md\cdot\mc_m}\bigg)\bigg(\frac{\mx-\mc_m}{|\mx-\mc_m|}\cdot\mt(\mc_m)\bigg)J_1(k|\mx-\mc_m|).
\end{align*}
Hence, we can obtain (\ref{Structure2}) by applying H{\"o}lder's inequality. This completes the proof.
\end{proof}

\begin{rem}\label{Observation2}
Based on the existing results, $\mathcal{I}(\mx)$ is expected to exhibit peaks of magnitude of $1$ at the location $\mx=\mc_m\in\Sigma_m$ and of small magnitudes at $\mx\notin\Sigma_m$. However, based on the identified structure (\ref{Structure2}), value of $\mathcal{I}(\mx)$ is close to $1$ when $|\Phi(\mx,\mc_m,\md)|$ reaches its maximum value. This means that identified location is highly depending on the direction of propagation $\md$ and unit tangential direction $\mt(\mc_m)$, i.e., rotation of $\Sigma_m$. For example, if $\md=-\mt(\mc_m)$ and $\mx-\mc_m$ is parallel to $\mt(\mc_m)$ then, $|\Phi(\mx,\mc_m,\md)|$ does not have its maximum value at $\mx=\mc_m$. This is the reason why inaccurate location of $\Sigma_m$ is identified via the direct sampling method.
\end{rem}

Now, let us consider the effect of $\Phi_2(\mx,\md)$ in (\ref{secondterm}).
\begin{cor}\label{CorollaryError}Assume that $M=1$. Then, if $\mx$ is close to $\mc$ such that $0<k|\mx-\mc|\ll\sqrt{2}$ then
\begin{equation}\label{error1}
|\Phi_2(\mx,\md)|\ll|\Phi_1(\mx,\md)|
\end{equation}
and if $\mx$ is far away from $\mc$ such that $k|\mx-\mc|\gg0.75$ then
\begin{equation}\label{error2}
|\Phi_1(\mx,\md)|,|\Phi_2(\mx,\md)|\longrightarrow0+
\end{equation}
for all $\mx\in\mathbb{R}^2$.
\end{cor}
\begin{proof}
Let $\Gamma(n)$ denotes the Gamma function and $\mx$ is close to $\mc$ such that $0<k|\mx-\mc|\ll\sqrt{2}$. Then, based on the asymptotic form of Bessel function
\[J_n(x)\approx\frac{1}{\Gamma(n+1)}\left(\frac{x}{2}\right)^n,\]
we can observe that
\[|\Phi_1(\mx,\md)|=\left|\frac{2}{\ln(\ell/2)}e^{ik\md\cdot\mc}J_0(k|\mx-\mc|)\right|\leq\left|\frac{2}{\ln(\ell/2)}\frac{1}{\Gamma(1)}\right|=\left|\frac{2}{\ln(\ell/2)}\right|.\]
Since $k\ell\rightarrow0+$,
\begin{align*}
|\Phi_2(\mx,\md)|&=\left|k^2\ell^2\bigg((\md\cdot\mt(\mc))e^{ik\md\cdot\mc}\bigg)\bigg(\frac{\mx-\mc}{|\mx-\mc|}\cdot\mt(\mc)\bigg)J_1(k|\mx-\mc|)\right|\\
&\leq (k\ell)^2|J_1(k|\mx-\mc|)|\approx\frac{\pi^2}{\Gamma(2)}\frac{k|\mx-\mc|}{2}\ll\frac{(k\ell)^2}{\sqrt{2}}\longrightarrow0.
\end{align*}
Hence, the value of $|\Phi_2(\mx,\md)|$ can be negligible and we can conclude (\ref{error1}) holds.

Now, we let $\mx$ is far away from $\mc$ such that $k|\mx-\mc|\gg0.75$. Then based on the asymptotic form of Bessel function
\[J_n(x)\approx\sqrt{\frac{2}{\pi x}}\left(\cos\left(x-\frac{n\pi}{2}-\frac{\pi}{4}\right)+\mathcal{O}\left(\frac{1}{|x|}\right)\right),\]
we can observe that
\[|\Phi_1(\mx,\md)|\leq\frac{2\sqrt{2}}{|\ln(\ell/2)|\sqrt{k|\mx-\mc|}}\ll\frac{8\sqrt{2}}{3|\ln(\ell/2)|}\longrightarrow0+\]
and
\[|\Phi_2(\mx,\md)|\frac{(k\ell)^2\sqrt{2}}{\sqrt{k|\mx-\mc|}}\ll\frac{4\sqrt{2}(k\ell)^2}{3}\longrightarrow0+.\]
Hence, we can conclude (\ref{error2}) holds. This completes the proof.
\end{proof}

\begin{rem}
  Fortunately, based on Corollary \ref{CorollaryError}, the term (\ref{secondterm}) does not significantly contribute to the imaging performance. Hence, we can say that identified location is close to the true location $\mc_m$. This means that it can be regarded as good initial guess and exact location of $\Sigma_m$ can be retrieved via the Newton-type iteration scheme, two-stage method or level-set strategy, refer to \cite{K,ADIM,DL,IJZ3,KS2,PL3}.
\end{rem}

\section{Two methods of improvement}\label{sec:4}
From now on, we investigate two methods of improvement for obtaining better results than the traditional direct sampling method. The first one is the application of multiple number of incident directions and the second one is the multi-frequency based imaging technique. Notice that since we try to improve imaging performance only, we do not consider the shift phenomenon considered in Section \ref{sec3-3}.

\subsection{Improvement of indicator function: multiple directions of incident fields}\label{sec4-1}
First, we consider an improvement of the direct sampling method using a set of measured far-field pattern data:
\[\mathcal{F}:=\set{\psi_\infty(\vt_n;\md_l):n = 1, 2, \cdots, N,~l = 1, 2, \cdots,L},\]
where we assume that the total number $L$ of the incident fields is small and set
\[\md_l = [\cos\theta_l,\sin\theta_l]^T=\bigg[\cos\frac{2\pi l}{L},\sin\frac{2\pi l}{L}\bigg]^T.\]

In several researches \cite{IJZ1,IJZ2,LZ}, an indicator function of the direct sampling method with a few number of incident fields is designed as follows:
\begin{equation}\label{ImprovedImagingFunction}
\mathcal{I}_{\mathrm{IF}}(\mx):=\max_{\mx\in\mathbb{R}^2}\set{\mathcal{I}(\mx;l):l = 1, 2, \cdots, L},
\end{equation}
where $\mathcal{I}(\mx;l)$ with incident direction $\md=\md_l$ is
\[\mathcal{I}(\mx;l):=\frac{|\langle \psi_\infty(\vt_n,\md_l),e^{-ik\vt_n\cdot\mx}\rangle|}{||\psi_\infty(\vt_n,\md_l)||_{L^2(\mathbb{S}^1)}||e^{-ik\vt_n\cdot\mx}||_{L^2(\mathbb{S}^1)}}.\]
Unfortunately, it is still difficult to identify cracks with relatively small lengths. Due to this reason, we suggest an alternative indicator function $\mathcal{I}_{\mathrm{AIF}}(\mx,L)$ for improving (\ref{ImprovedImagingFunction}):
\begin{equation}\label{ImprovedImagingFunction1}
  \mathcal{I}_{\mathrm{AIF}}(\mx,L):=\frac{|\Psi(\mx,L)|}{\displaystyle\max_{\mx\in\mathbb{R}^2}|\Psi(\mx,L)|},
\end{equation}
where
\begin{equation}\label{ImprovingFactor}
\Psi(\mx,L):=\sum_{l = 1}^{L}e^{-ik\md_l\cdot\mx}\langle \psi_\infty(\vt_n;\md_l),e^{-ik\vt_n\cdot\mx}\rangle.
\end{equation}

Theoretical reason of improvement is derived as follows.

\begin{thm}\label{TheoremImprovement1}
Assume that total number of observation direction $N$ is sufficiently large and incident direction $L$ is small. Then, by letting $\mc_m-\mx=|\mc_m-\mx|[\cos\varphi_m,\sin\varphi_n]^T$, $\mathcal{I}_{\mathrm{AIF}}(\mx,L)$ can be represented as follows:
\begin{equation}\label{StructureImprovement1}
\mathcal{I}_{\mathrm{AIF}}(\mx,L)=\frac{|\Psi_1(\mx)+\Psi_2(\mx,L)|}{\displaystyle\max_{\mx\in\mathbb{R}^2}|\Psi_1(\mx)+\Psi_2(\mx,L)|},
\end{equation}
where
\begin{align*}
\Psi_1(\mx)&=\sum_{m=1}^{M}\frac{2\pi^2}{\ln(\ell_m/2)}J_0(k|\mx-\mc_m|)^2\\
\Psi_2(\mx,L)&=\frac{1}{L}\sum_{m=1}^{M}\sum_{l = 1}^{L}\sum_{s=1}^{\infty}\frac{2\pi^2i^s}{\ln(\ell_m/2)}J_0(k|\mx-\mc_m|)J_s(k|\mx-\mc_m|)\cos(s(\varphi_m-\theta_l)).
\end{align*}
\end{thm}
\begin{proof}
Based on the proof of Theorem \ref{TheoremStructure1}, $\Psi(\mx,L)$ can be written
\begin{align*}
\Psi(\mx,L)&=\sum_{l = 1}^{L}e^{-ik\md_l\cdot\mx}\langle \psi_\infty(\vt_n;\md_l),e^{-ik\vt_n\cdot\mx}\rangle=\sum_{l = 1}^{L}e^{-ik\md_l\cdot\mx}\left(\sum_{m=1}^{M}\frac{(2\pi)^2}{\ln(\ell_m/2)}e^{ik\md_l\cdot\mc_m}J_0(k|\mx-\mc_m|)\right)\\
&=\sum_{m=1}^{M}\frac{(2\pi)^2}{\ln(\ell_m/2)}J_0(k|\mx-\mc_m|)\left(\sum_{l = 1}^{L}e^{ik\md_l\cdot(\mc_m-\mx)}\right).
\end{align*}
Since $L$ is not sufficiently large, we cannot apply Lemma \ref{TheoremBessel1}. Instead of this, applying Jacobi-Anger expansion
\begin{equation}\label{Jacobi-Anger}
e^{iz\cos\phi}=J_0(z)+2\sum_{s=1}^{\infty}i^sJ_s(z)\cos(s\phi)
\end{equation}
yields
\begin{align*}
\sum_{l = 1}^{L}e^{ik\md_l\cdot(\mc_m-\mx)}&=\sum_{l = 1}^{L}e^{ik|\mc_m-\mx|\cos(\varphi_m-\theta_l)}=\sum_{l = 1}^{L}\left(J_0(k|\mc_m-\mx|)+2\sum_{s=1}^{\infty}i^sJ_s(k|\mc_m-\mx|)\cos(s(\varphi_m-\theta_l))\right).
\end{align*}
Hence, we arrive
\[\Psi(\mx,L)=\sum_{m=1}^{M}\frac{(2\pi)^2}{\ln(\ell_m/2)}J_0(k|\mx-\mc_m|)\left(LJ_0(k|\mc_m-\mx|)+2\sum_{l = 1}^{L}\sum_{s=1}^{\infty}i^sJ_s(k|\mc_m-\mx|)\cos(s(\varphi_m-\theta_l))\right)\]
and correspondingly, we can obtain the result (\ref{StructureImprovement1}). This completes the proof.
\end{proof}

\begin{rem}\label{Observation3}Based on the structures (\ref{Structure1}) and (\ref{StructureImprovement1}), we can easily observe that
\begin{equation}\label{Compare1}
\mathcal{I}_{\mathrm{IF}}(\mx)\propto|J_0(k|\mx-\mc_m|)|\quad\mbox{and}\quad\mathcal{I}_{\mathrm{AIF}}(\mx,L)\propto J_0(k|\mx-\mc_m|)^2.
\end{equation}
Two-dimensional plot for (\ref{Compare1}) is shown in Figure \ref{FigureCompare1}. By considering the oscillation pattern, we can easily observe that $\mathcal{I}_{\mathrm{AIF}}(\mx,L)$ yields better images owing to less oscillation than $\mathcal{I}_{\mathrm{IF}}(\mx)$ does. Furthermore, based on $\Psi_2(\mx,L)$, unexpected artifacts in the map of $\mathcal{I}_{\mathrm{AIF}}(\mx,L)$ are mitigated when $L$ is sufficiently large. This result indicates why increasing total number of incident fields guarantee good results.
\end{rem}

\begin{figure}[h]
  \begin{center}
    \includegraphics[width=0.495\textwidth]{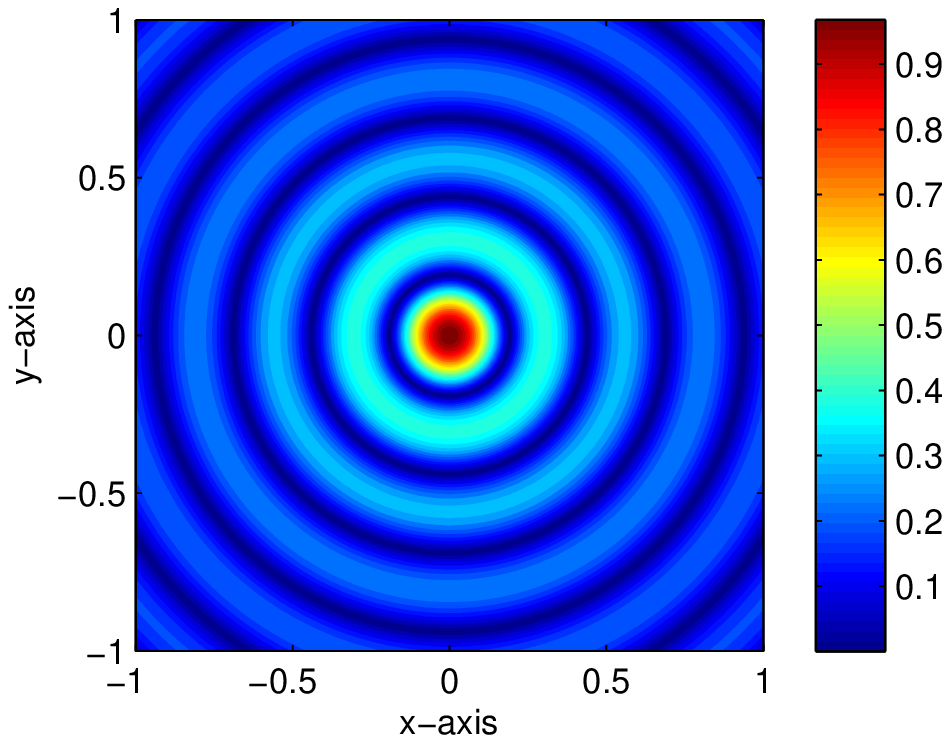}
    \includegraphics[width=0.495\textwidth]{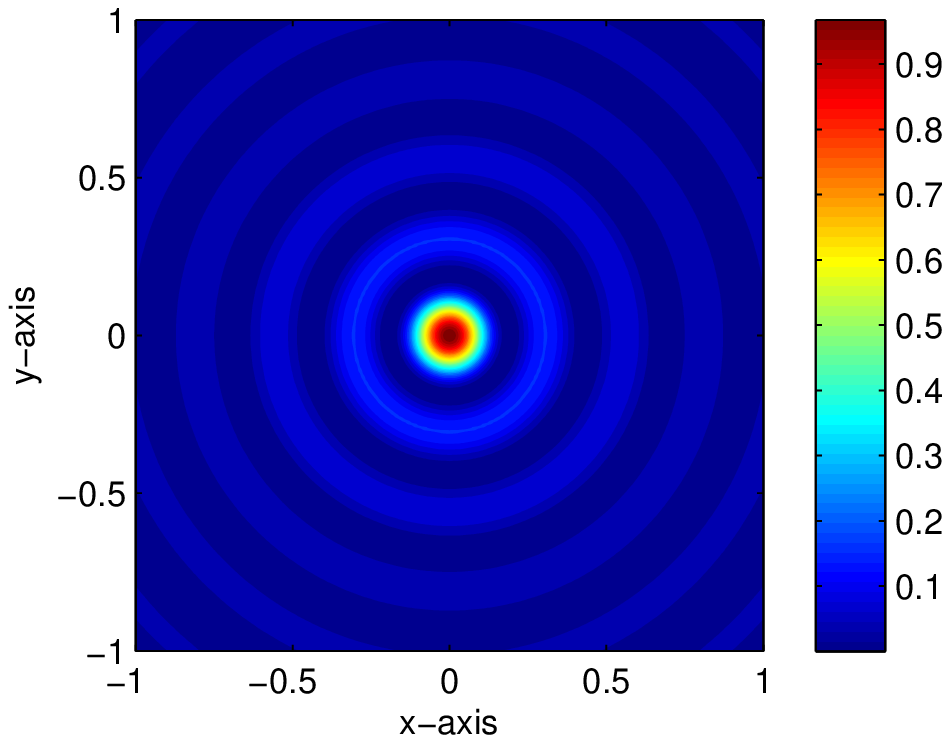}
    \caption{\label{FigureCompare1}Two-dimensional plot of $|J_0(k|\mx|)|$ (left) and $J_0(k|\mx|)^2$ (right) for $k=2\pi/0.5$.}
  \end{center}
\end{figure}

\begin{ex}[Comparing $\mathcal{I}_{\mathrm{IF}}(\mx)$ and $\mathcal{I}_{\mathrm{AIF}}(\mx,L)$]
Figure \ref{DSM-IF-AIF} shows the maps of $\mathcal{I}_{\mathrm{IF}}(\mx)$ and $\mathcal{I}_{\mathrm{AIF}}(\mx,L)$ for $L=2,3,4$ incident directions. Based on these results, it is clearly difficult to discriminate the locations of $\Sigma_2$ and $\Sigma_3$ from the map of $\mathcal{I}_{\mathrm{AIF}}(\mx,L)$ when $L=2$ and $3$ due to the appearance of abundant artifacts. When $L=4$ directions are used, it is possible to identify the $\Sigma_m$ locations from the map of $\mathcal{I}_{\mathrm{AIF}}(\mx,L)$; however, it still contains some artifacts with large magnitude.

It is interesting to observe that although a huge number of artifacts disturbs imaging performance, it is possible to identify locations of $\Sigma_1$ and $\Sigma_3$ from the map of $\mathcal{I}_{\mathrm{IF}}(\mx)$ when $L=2$ and $4$. Hence, it is hard to say that $\mathcal{I}_{\mathrm{AIF}}(\mx,L)$ is an improved version of $\mathcal{I}_{\mathrm{IF}}(\mx)$ when $L$ is small. This supports discussion in Remark \ref{Observation3}.
\end{ex}

\begin{figure}[h]
  \begin{center}
    \includegraphics[width=0.495\textwidth]{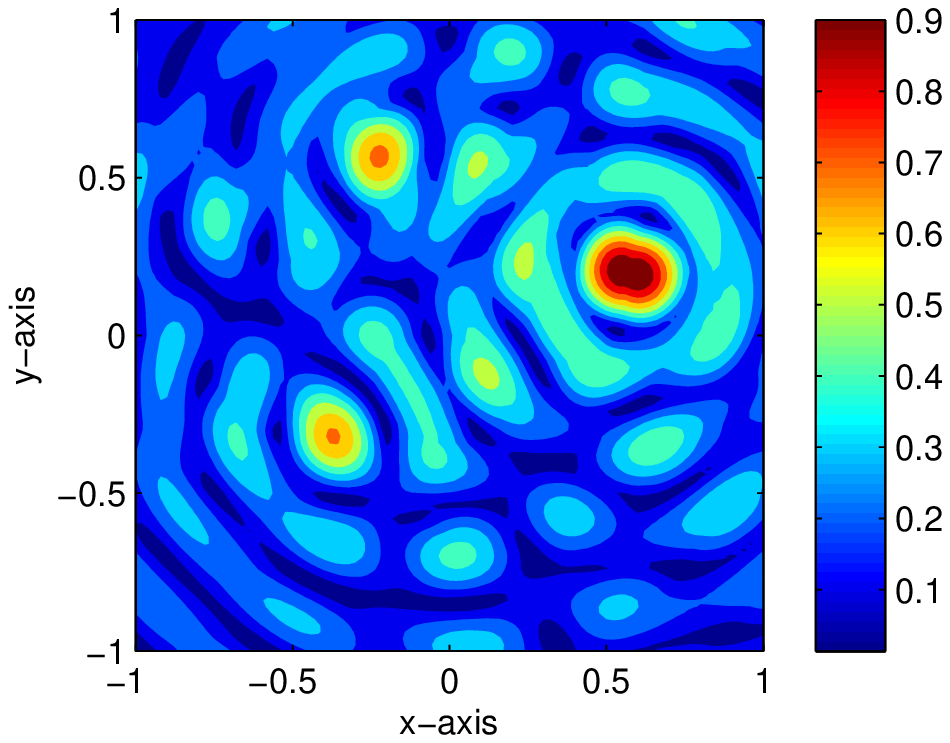}
    \includegraphics[width=0.495\textwidth]{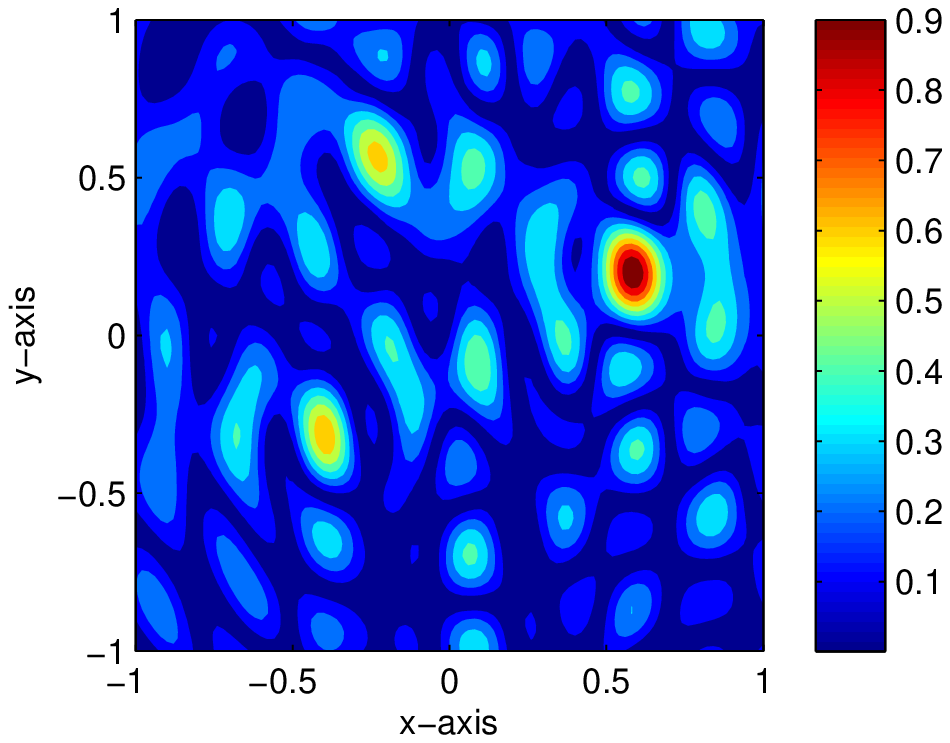}\\
    \includegraphics[width=0.495\textwidth]{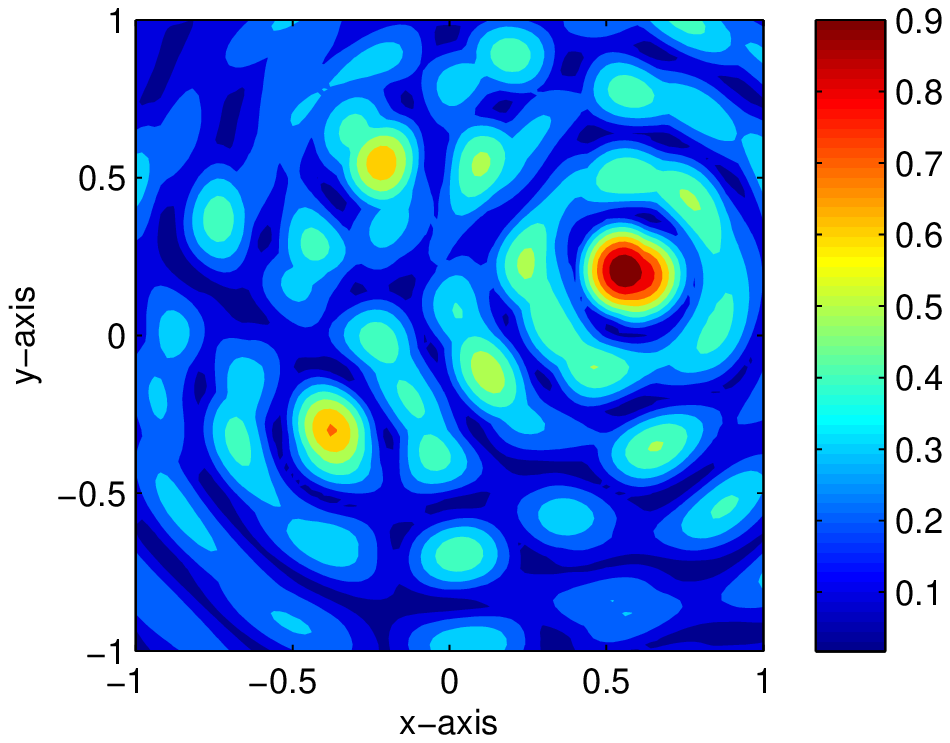}
    \includegraphics[width=0.495\textwidth]{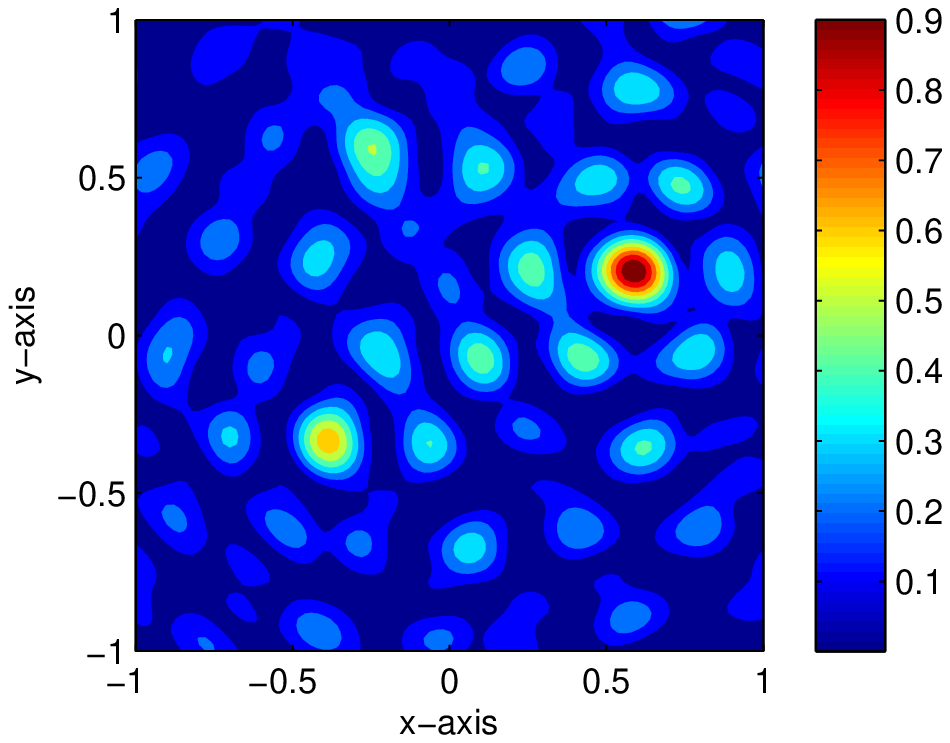}\\
    \includegraphics[width=0.495\textwidth]{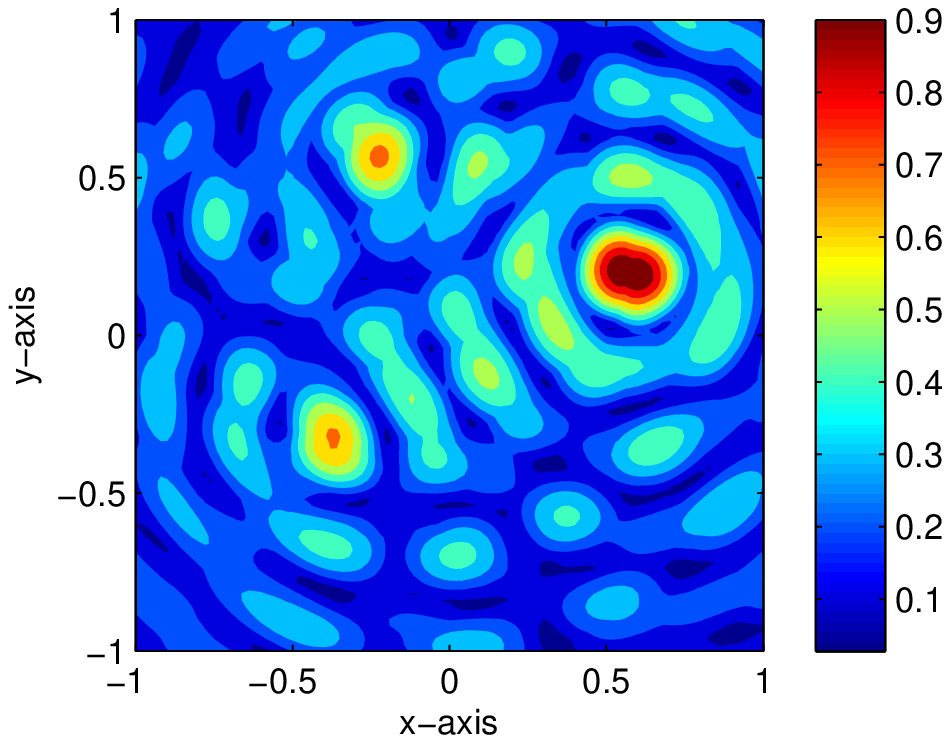}
    \includegraphics[width=0.495\textwidth]{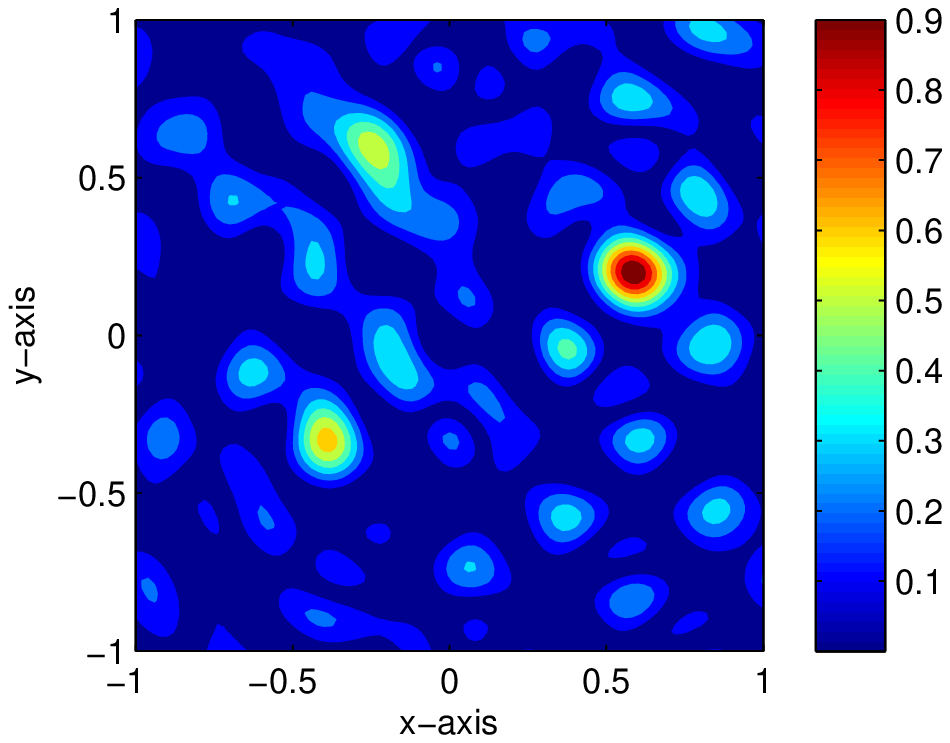}
    \caption{\label{DSM-IF-AIF}Maps of $\mathcal{I}_{\mathrm{IF}}(\mx)$ (left column) and $\mathcal{I}_{\mathrm{AIF}}(\mx,L)$ (right column) for $L=2$ (top), $L=3$ (middle), and $L=4$ (bottom).}
  \end{center}
\end{figure}

\begin{ex}[Influence of total number of incident directions]
Figure \ref{DSM-AIF} shows the maps of $\mathcal{I}_{\mathrm{AIF}}(\mx,L)$ for $L=5,6,7,8$ incident directions. Opposite to the results in Figure \ref{DSM-IF-AIF}, map of $\mathcal{I}_{\mathrm{AIF}}(\mx,L)$ yields satisfactory results when the total number of $L$ increases such that $L\geq5$. It is worth observing that the locations of all $\Sigma_m$ were well-identified via the map of $\mathcal{I}_{\mathrm{AIF}}(\mx,6)$; however, due to the existence of some artifacts with large magnitude, it is still difficult to discriminate the locations of $\Sigma_2$ and $\Sigma_3$.
\end{ex}

\begin{figure}[h]
  \begin{center}
    \includegraphics[width=0.495\textwidth]{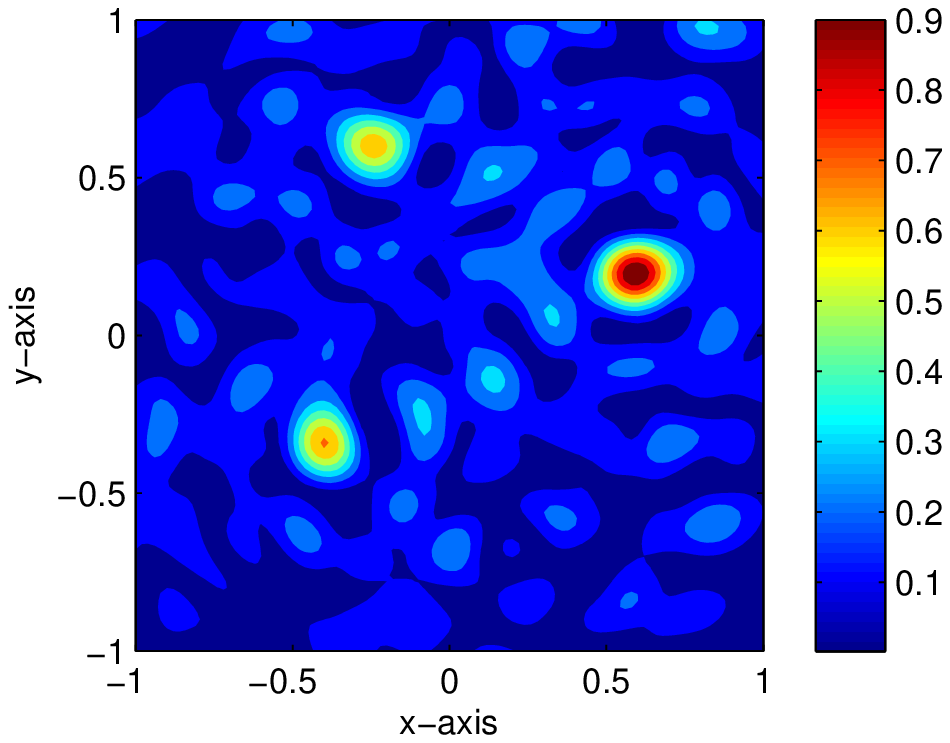}
    \includegraphics[width=0.495\textwidth]{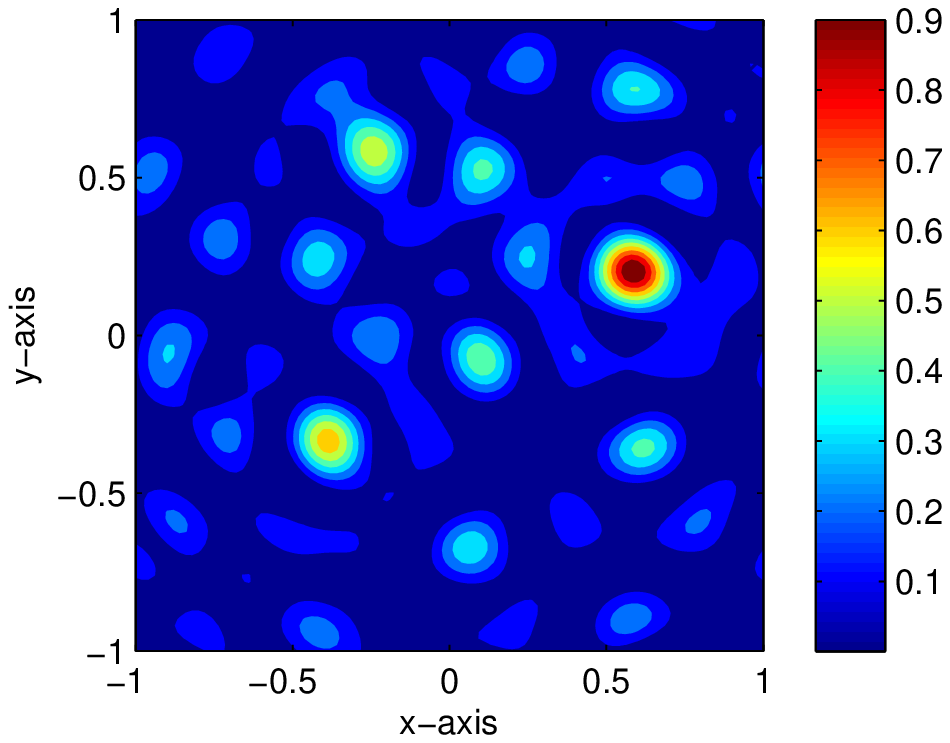}\\
    \includegraphics[width=0.495\textwidth]{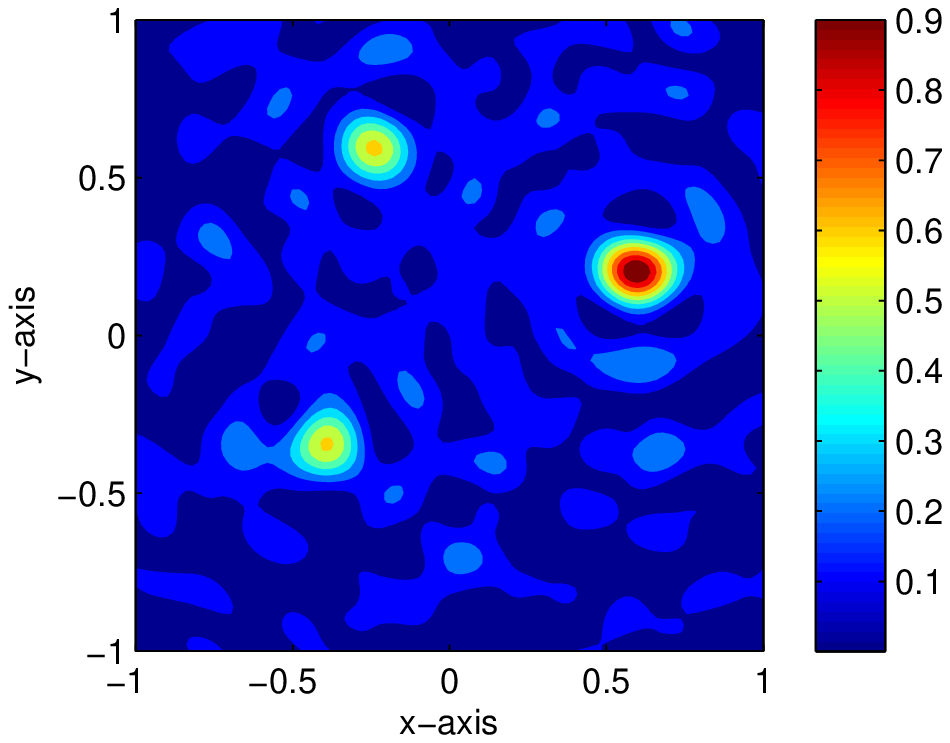}
    \includegraphics[width=0.495\textwidth]{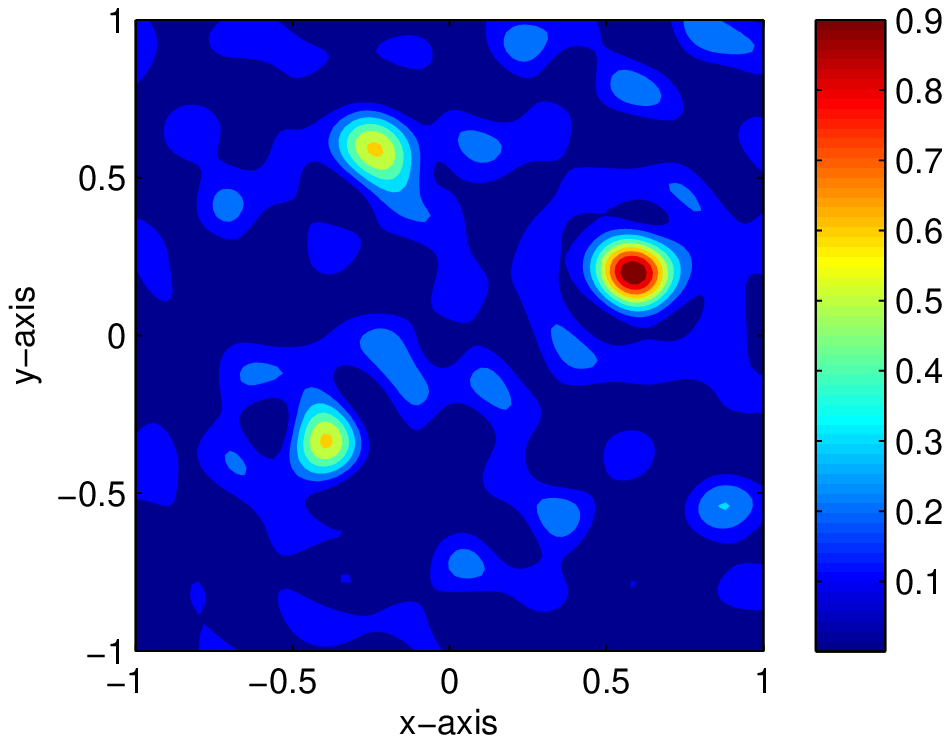}
    \caption{\label{DSM-AIF}Maps of $\mathcal{I}_{\mathrm{AIF}}(\mx,L)$ for $L=5$ (top, left), $L=6$ (top, right), $L=7$ (bottom, left), and $L=8$ (bottom, right).}
  \end{center}
\end{figure}

\subsection{Improvement of indicator function: application of multiple frequencies}\label{sec4-2}
One of famous and useful method is application of multiple wavenumbers $k_f$, $f=1,2,\cdots,F$. In this section, we consider the following multi-frequency indicator function
\begin{equation}\label{ImprovedImagingFunction2}
\mathcal{I}_{\mathrm{MIF}}(\mx,F):=\frac{|\mathcal{I}_{\mathrm{MF}}(\mx,k_f)|}{\displaystyle\max_{\mx\in\mathbb{R}^2}|\mathcal{I}_{\mathrm{MF}}(\mx,k_f)|},
\end{equation}
where based on (\ref{ImprovingFactor}), $\mathcal{I}_{\mathrm{MF}}(\mx,k_f)$ is given by
\[\mathcal{I}_{\mathrm{MF}}(\mx,k_f):=\sum_{f=1}^{F}e^{-ik_f\md\cdot\mx}\langle \psi_\infty(\vt_n,\md,k_f),e^{-ik_f\vt_n\cdot\mx}\rangle.\]
Here, $\psi_\infty(\vt_n,\md,k_f)$ denotes the far-field pattern (\ref{FarFieldPattern}) at wavenumber $k_f$ and $F$ denotes total number of applied wavenumber. Similar to the results \cite{P-SUB3,AGKPS,HHSZ,P-TD1,P-TD5}, application of multiple frequencies guarantees better imaging results than the single frequency. Theoretical reason of this phenomenon is follows. In this case, we apply $L=1$ number of incident field.

\begin{thm}\label{TheoremImprovement2}
Assume that total number of observation direction $N$ and applied number of wavenumber $F$ are sufficiently large. Then, by letting $\mc_m-\mx=|\mc_m-\mx|[\cos\varphi_m,\sin\varphi_n]^T$ and $k_1<k_2<\cdots<k_F$, $\mathcal{I}_{\mathrm{MIF}}(\mx,F)$ can be represented as follows:
\begin{equation}\label{StructureImprovement2}
\mathcal{I}_{\mathrm{MIF}}(\mx,F)=\frac{|\Psi_3(\mx)+\Psi_4(\mx)|}{\displaystyle\max_{\mx\in\mathbb{R}^2}|\Psi_3(\mx)+\Psi_4(\mx)|},
\end{equation}
where
\begin{align*}
\Psi_3(\mx)&=\sum_{m=1}^{M}\frac{(2\pi)^2}{\ln(\ell_m/2)}\left[k_F\bigg(J_0(k_F|\mx-\mc_m|)^2+J_1(k_F|\mx-\mc_m|)^2\bigg)-k_1\bigg(J_0(k_1|\mx-\mc_m|)^2+J_1(k_1|\mx-\mc_m|)^2\bigg)\right]\\
\Psi_4(\mx)&=\sum_{m=1}^{M}\frac{(2\pi)^2}{\ln(\ell_m/2)}\int_{k_1}^{k_F}\bigg(J_1(k|\mx-\mc_m|)^2+2\sum_{s=1}^{\infty}i^sJ_0(k|\mx-\mc_m|)J_s(k|\mx-\mc_m|)\cos(s(\varphi_m-\theta))\bigg)dk.
\end{align*}
\end{thm}
\begin{proof}
Applying Jacobi-Anger expansion (\ref{Jacobi-Anger}) to $\mathcal{I}_{\mathrm{MF}}(\mx,k_f)$ yields
\begin{align*}
\mathcal{I}_{\mathrm{MF}}(\mx,k_f)&=\sum_{f=1}^{F}e^{-ik_f\md\cdot\mx}\langle \psi_\infty(\vt_n,\md,k_f),e^{-ik_f\vt_n\cdot\mx}\rangle=\sum_{f=1}^{F}\sum_{m=1}^{M}\frac{(2\pi)^2}{\ln(\ell_m/2)}e^{ik_f\md\cdot(\mc_m-\mx)}J_0(k_f|\mx-\mc_m|)\\
&=\sum_{m=1}^{M}\frac{(2\pi)^2}{\ln(\ell_m/2)}\sum_{f=1}^{F}\left(J_0(k_f|\mc_m-\mx|)+2\sum_{s=1}^{\infty}i^sJ_s(k_f|\mc_m-\mx|)\cos(s(\varphi_m-\theta))\right)J_0(k_f|\mx-\mc_m|)\\
&\approx\sum_{m=1}^{M}\frac{(2\pi)^2}{\ln(\ell_m/2)}\frac{1}{k_F-k_1}\int_{k_1}^{k_F}\left(J_0(k|\mx-\mc_m|)^2+2\sum_{s=1}^{\infty}i^sJ_0(k|\mx-\mc_m|)J_s(k|\mx-\mc_m|)\cos(s(\varphi_m-\theta))\right)dk.
\end{align*}
Based on an indefinite integral of Bessel function
\[\int J_0(x)^2dx=x\bigg(J_0(x)^2 +J_1(x)^2\bigg)+\int J_1(x)^2dx:=x\Lambda(x)+\int J_1(x)^2dx,\]
we can obtain
\begin{multline*}
\mathcal{I}_{\mathrm{MF}}(\mx,k_f)\approx\sum_{m=1}^{M}\frac{(2\pi)^2}{\ln(\ell_m/2)}\left[\frac{k_F}{k_F-k_1}\Lambda(k_F|\mx-\mc_m|)-\frac{k_1}{k_F-k_1}\Lambda(k_1|\mx-\mc_m|)\right.\\
\left.+\frac{1}{k_F-k_1}\int_{k_1}^{k_F}\bigg(J_1(k|\mx-\mc_m|)^2+2\sum_{s=1}^{\infty}i^sJ_0(k|\mx-\mc_m|)J_s(k|\mx-\mc_m|)\cos(s(\varphi_m-\theta))\bigg)dk\right].
\end{multline*}
Therefore, (\ref{StructureImprovement2}) derived. This completes the proof.
\end{proof}

\begin{rem}\label{Observation4}Based on results in Theorem (\ref{TheoremImprovement1}) and (\ref{TheoremImprovement2}), we can easily observe that
\begin{equation}\label{Compare2}
\mathcal{I}_{\mathrm{AIF}}(\mx,L)\propto J_0(k|\mx-\mc_m|)^2\quad\mbox{and}\quad\mathcal{I}_{\mathrm{MIF}}(\mx,F)\propto\left|\frac{k_F}{k_F-k_1}\Lambda(k_F|\mx-\mc_m|)-\frac{k_1}{k_F-k_1}\Lambda(k_1|\mx-\mc_m|)\right|.
\end{equation}
Two-dimensional plot for (\ref{Compare2}) is shown in Figure \ref{FigureCompare2}. Similar to the Remark \ref{Observation3}, $\mathcal{I}_{\mathrm{MIF}}(\mx,F)$ will yield better results owing to less oscillation than $\mathcal{I}_{\mathrm{AIF}}(\mx,L)$ does if total number of applied frequencies $F$ and incident directions $L$ are large and small, respectively. If $F$ and $L$ are sufficiently large, it is hard to compare the imaging performance because the terms $\Psi_2(\mx,L)$ of (\ref{StructureImprovement1}) and $\Psi_4(\mx)$ of (\ref{StructureImprovement2}) can be disregarded.
\end{rem}

\begin{figure}[h]
  \begin{center}
    \includegraphics[width=0.495\textwidth]{Bessel2.eps}
    \includegraphics[width=0.495\textwidth]{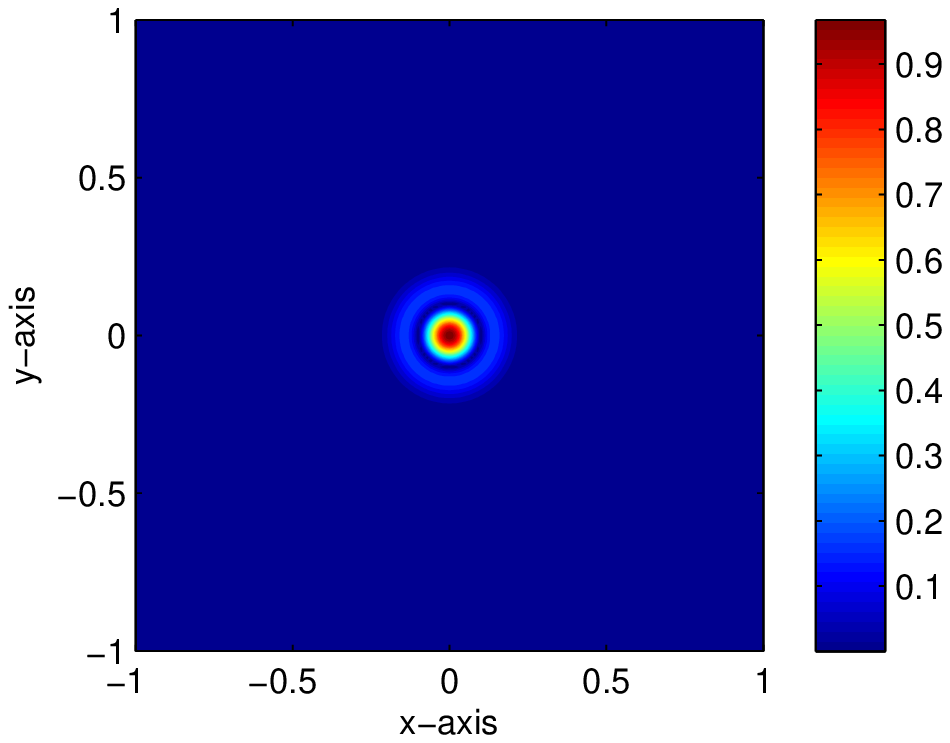}
    \caption{\label{FigureCompare2}Two-dimensional plot of $J_0(k|\mx|)^2$ for $k=2\pi/0.5$ (left) and $\left|k_F\Lambda(k_F|\mx|)-k_1\Lambda(k_1|\mx|)\right|/(k_F-k_1)$ for $k_1=2\pi/0.7$ and $k_F=2\pi/0.3$ (right).}
  \end{center}
\end{figure}

\begin{ex}[Influence of total number of frequencies]
Now, we perform numerical simulations for supporting Theorem \ref{TheoremImprovement2}. In this example, the wavelengths $\lambda_f$ are uniformly distributed in the interval $[\lambda_1,\lambda_F]$. Figure \ref{DSM-MIF} shows maps of $\mathcal{I}_{\mathrm{MIF}}(\mx,F)$ for $F=3,5,7,10$. Based on the results, $F=5$ is sufficient for obtaining a good result. Notice that as we discussed in Remark \ref{Observation4} and simulation results, increasing $F$ yields more better image.
\end{ex}

\begin{figure}[h]
  \begin{center}
    \includegraphics[width=0.495\textwidth]{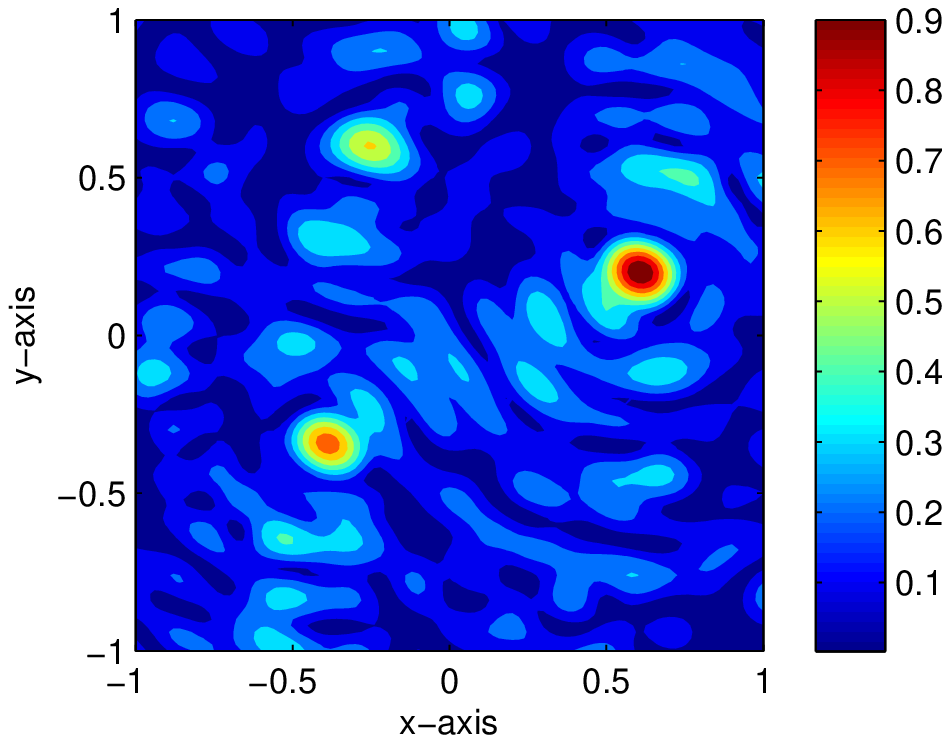}
    \includegraphics[width=0.495\textwidth]{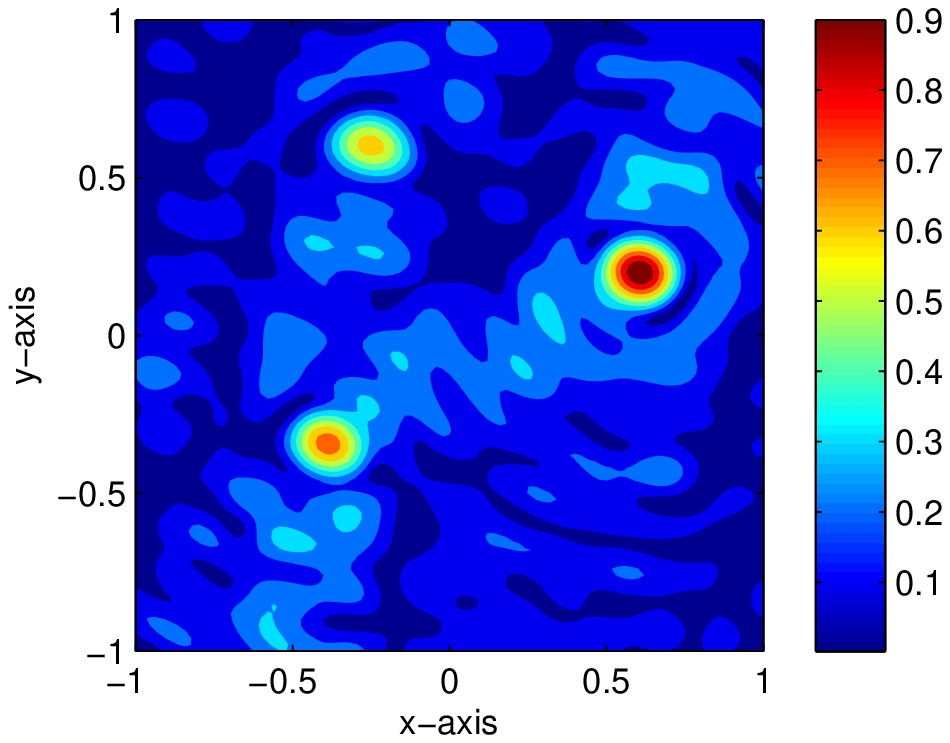}\\
    \includegraphics[width=0.495\textwidth]{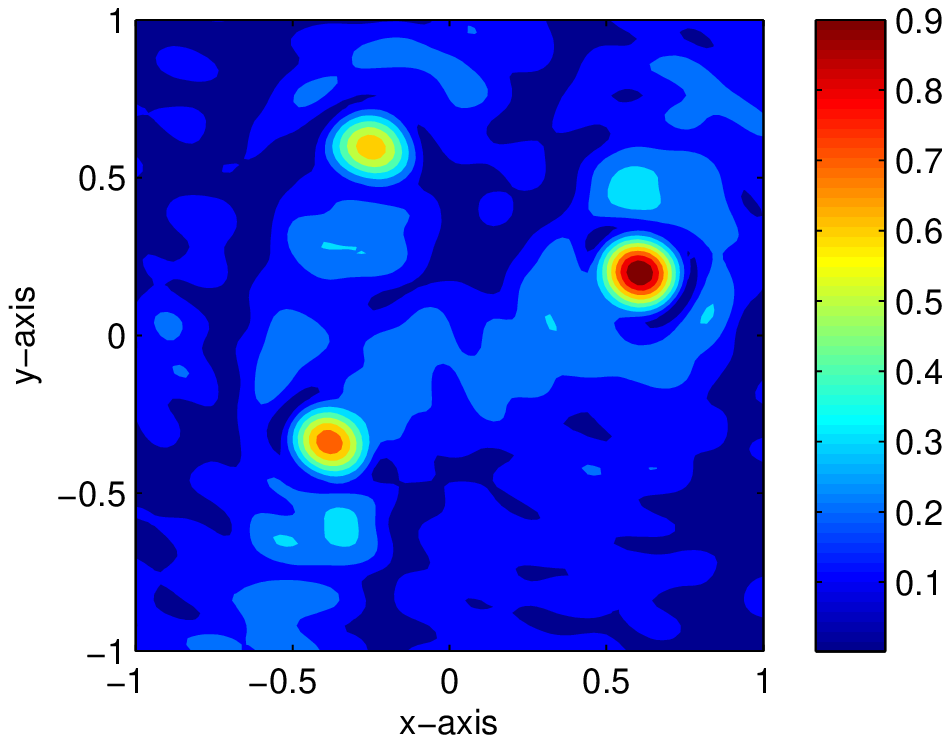}
    \includegraphics[width=0.495\textwidth]{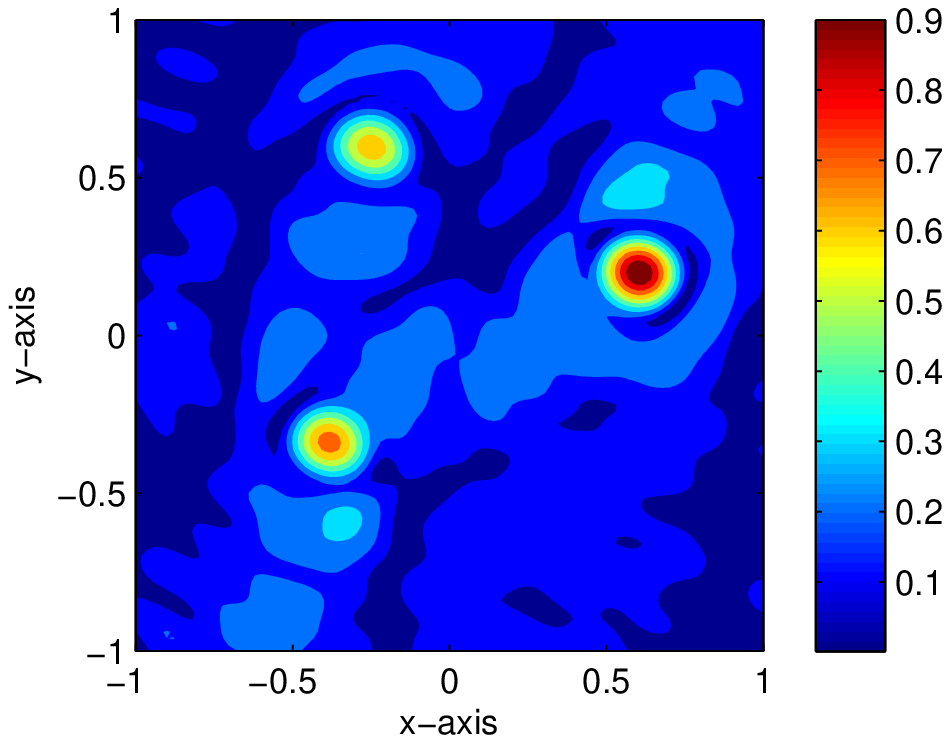}
    \caption{\label{DSM-MIF}Maps of $\mathcal{I}_{\mathrm{MIF}}(\mx,F)$ for $F=3$ (top, left), $F=5$ (top, right), $F=7$ (bottom, left), and $F=10$ (bottom, right).}
  \end{center}
\end{figure}


\section{Conclusion}\label{sec:5}
In this contribution, we have considered the direct sampling method for imaging cracks with small length. We investigated that the indicator function of direct sampling method can be represented by the Bessel function of order zero and one, incident direction, and unit tangential at crack. Based on the investigated representation of indicator function, we explained why detection performance depends on crack length, selection of direction of propagation, and rotation of crack.

Based on the investigated structure of indicator function, we proposed two methods for improving imaging performance. To prove the fact of enhancement, we investigated that proposed indicator functions can be represented by an infinite series of Bessel functions. Several simulation results were exhibited to support our investigation and motivate further research. We have considered the imaging of small cracks in this study, extension to arc-like cracks would be a forthcoming work. Furthermore, following \cite{PKLS}, application of direct sampling method from $S-$parameter data will be a remarkable research subject. Finally, following \cite{IJZ2}, extending the problem to three dimensions would also be an interesting problem.

\section*{Acknowledgement}
This research was supported by the Basic Science Research Program of the National Research Foundation of Korea (NRF) funded by the Ministry of Education (No. NRF-2017R1D1A1A09000547).

\bibliographystyle{elsarticle-num-names}
\bibliography{../../../References}
\end{document}